\documentclass[11pt]{amsart}

\usepackage[margin=1.2in]{geometry}
\usepackage{amsmath,amssymb,amsthm,mathtools}
\usepackage{mathrsfs}
\usepackage{graphicx}
\usepackage{xcolor}
\usepackage[utf8]{inputenc}
\usepackage[T1]{fontenc}

\usepackage[
    colorlinks=true,
    linkcolor=blue,
    citecolor=blue,
    urlcolor=blue
]{hyperref}

\usepackage[nameinlink,capitalize,noabbrev]{cleveref}

\newtheorem{theorem}{Theorem}[section]
\newtheorem{proposition}[theorem]{Proposition}
\newtheorem{lemma}[theorem]{Lemma}
\newtheorem{corollary}[theorem]{Corollary}

\theoremstyle{definition}
\newtheorem{definition}[theorem]{Definition}

\theoremstyle{remark}
\newtheorem{remark}[theorem]{Remark}

\newcommand{\N}{\mathbb{N}}
\newcommand{\tr}{\text{tr}}
\newcommand{\Sym}{\text{Sym}}

\newcommand{\Q}{\mathbb{Q}}
\newcommand{\Z}{\mathbb{Z}}
\newcommand{\R}{\mathbb{R}}
\newcommand{\C}{\mathbb{C}}

\newcommand{\Kl}{\text{Kl}}
\newcommand{\GL}{\text{GL}}
\newcommand{\SL}{\text{SL}}
\newcommand{\A}{\mathbb{A}}
\newcommand{\Res}{\text{Res}}
\newcommand{\wt}{\text{wt}}

\newcommand{\X}{Z_{\infty}^{+}G(\Q) \backslash G(\A)}
\newcommand{\qbinom}[2]{\genfrac{[}{]}{0pt}{}{#1}{#2}}
\newcommand{\cusp}{\text{cusp}}
\newcommand{\cont}{\text{cont}}
\newcommand{\XP}{Z_{\infty}^{+}P(\Q) \backslash G(\A)}
\newcommand{\Ell}{\text{ell}}
\newcommand{\hyp}{\text{hyp}}
\newcommand{\unip}{\text{unip}}
\newcommand{\F}{{}_2F_1}
\newcommand{\M}{\text{M}}
\newcommand{\vol}{\text{vol}}

\title{Jacquet-Zagier treatment of the beyond endoscopy trace formula for \(\GL_2\)}
\author{Pranjal Pandurang Warade}

\begin{document}

\begin{abstract}
We begin the study of the beyond endoscopic trace formula for \(\GL_2\) over \(\Q\) attached to any symmetric power representation \(\sigma_k\) of the dual group \(\GL_2(\C)\). For an adelic function that incorporates the \(L\)-functions \(L(s_B,\pi,\sigma_k)\) through the basic functions at all the finite places, we integrate the cuspidal kernel against a corresponding Eisenstein series \(E(g,s)\) and realize the trace formula as a residue at \(s=1\), replacing Arthur's truncation operation by a continuously deformed trace formula. We argue Poisson summation on the trace variable of the Hitchin-Steinberg base and show that the dominant term of the elliptic part admits meromorphic continuation to \(\mathfrak{R}(s_B) \geq 0\) with a pole of order \(k\) at \(s_B =1\).
\end{abstract}

\maketitle

\section{Introduction}
For an automorphic representation $\pi = \otimes_{\nu} \pi_{\nu}$ of a reductive group $G$ over a number field $F$, and a finite-dimensional representation $$\rho : \prescript{L}{}{G} \rightarrow \GL(V_{\rho})$$
of the $L$-group of $G$, Langlands defined the associated automorphic $L$-function as an Euler product of local factors \cite{borel1979automorphic}. More precisely, for a finite set $S$ of places containing the ramified and archimedean places, one considers the partial $L$-function given by $$L^S(s_B,\pi,\rho) = \prod_{\nu \not \in S} \det(1 - \rho(\alpha_{\pi_{\nu}}) q^{-s_B}_{\nu})^{-1} = \sum_{\substack{n \\ (n,S) = 1}} \frac{a_{\pi,\rho}(n)}{n^{s_B}}$$ where $\alpha_{\pi_{\nu}} \in \prescript{L}{}{G}$ denotes the local parameter of $\pi_{\nu}$. This Euler product converges absolutely in a right half-plane and hence defines a holomorphic function there. The $L$-function conjecture of Langlands \cite{Langlands1976OnTF} states that this function can be extended meromorphically to the whole complex plane, and should satisfy a functional equation similar to that of the Riemann $\zeta$-function.
\\
This conjecture is traditionally dealt via the functoriality conjectures \cite{functoriality}. Roughly speaking, if $G$ and $H$ are reductive groups over $F$, and $$\epsilon : \prescript{L}{}{H} \rightarrow \prescript{L}{}{G}$$ is a homomorphism of $L$-groups, then functoriality predicts a transfer of (packets) of automorphic representations of $H$ to (packets) of automorphic representations of $G$ such that the corresponding local parameters are compatible with $\epsilon$. In particular, the associated $L$-functions are expected to agree in the sense that $$L^S(s_B, \pi^H, \rho\circ\epsilon) = L^S(s_B, \pi, \rho)$$
By the theorem of Godement and Jacquet \cite{godement2006zeta}, the moremorphic continuation and functional equation of \(L^S(s_B,\pi,\rho)\) is known for $G = \GL_n$ and $\rho$ standard representation. Thus, having access to the functoriality conjecture will allow us to derive the meromorphic continuation and functional equation of some general automorphic $L$-functions from the special case of standard $L$-functions.
\\
One important subset of the functoriality program is endoscopy. Roughly, this is the case where $\prescript{L}{}{H}$ is essentially the centralizer of a semi-simple element in $\prescript{L}{}{G}$. Endoscopic transfer is now known in great generality through the cumulative efforts of many mathematicians over the past several decades. The main components of the proof include development of the  trace formula, due in large part to Arthur \cite{arthur2001}, together with the proof of the transfer conjecture \cite{waldspurger1995} and the fundamental lemma \cite{PMIHES_2010__111__1_0}. Despite these major advances, however, endoscopy covers only a limited part of the general functoriality conjecture, which remains widely open.
\\
In order to go beyond the endoscopic setting, Langlands proposed a new strategy based on a non-comparative use of the trace formula, now known as \emph{Beyond Endoscopy} \cite{Langlands2004}. The main aim of beyond endoscopy is to isolate those automorphic representations of $G$, unramified outside of $S$ that arise as functorial transfers  from other groups. In order to detect this functorial transfer,  Langlands proposed studying the partial $L$-functions $L^S(s_B, \pi ,\rho)$ for various irreducible representations $\rho$ of the dual group. The expectation is that if $L^S(s_B, \pi ,\rho)$ has a pole in $\mathfrak{R}(s_B) \geq 1$, the $\pi$ should be a functorial transfer. Conversely, if $\pi$ is a functorial transfer, then by the theorem of Chevally \cite{chevalley1968theorie}, it is possible to find $\rho$ such that $L^S(s_B,\pi,\rho)$ has a pole in $\mathfrak{R}(s_B) \geq 1$. 
\\
In his original paper on beyond endoscopy, Langlands formulated this approach in terms of arithmetic averages involving logarithmic derivatives of $L$-functions \cite{Langlands2004}.  Soon after this, P. Sarnak, in his letter to Langlands \cite{Sarnak}, suggested studying the $L$-functions themselves instead of their arithmetic sums. Although this suffers from the lack of formal additivity that makes the logarithmic derivative setup attractive, it has the advantage of inserting the $L$-functions directly into the trace formula and appears more amenable to analytic techniques such as Poisson summation. In this thesis, we follow the suggestion of Sarnak ahead.
\\
This approach has been carried out in the case of $G = \GL_2$ over $\Q$ for the standard representation of the dual by Ali Altug \cite{Altu__2015},\cite{altug2015endoscopytraceformula},\cite{altug2015endoscopytraceformulaII}. In this work, the main components were using an approximate functional equation to treat the volume terms, applying Poisson summation on the trace variable, and further studying the partial averages by applying Poisson summation again on the determinant variable. 
\\
Thus, establishing a general trace formula through the diagonal integral of a cuspidal kernel is central to Langlands program. Traditionally, for a nice enough function $\phi$ on $\GL_2(\A)$, the Selberg trace formula is obtained by writing the cuspidal kernel $K_0(g,g)$ as 
$$K_0^{\phi}(g,g) = K^{\phi}(g,g) - K_{\text{sp}}^{\phi}(g,g) - K_{\text{cont}}^{\phi}(g,g)$$
and then integrating along the diagonal. Although the cuspidal kernel is rapidly decreasing, the individual terms on the right are not, so one must regularize the resulting integrals, either by deleting small neighborhoods of the cusp forms from a fundamental domain or by truncating the kernel functions by subtracting their constant terms in such neighborhoods and then passing to a limit.
\\
An alternate approach towards trace formula was proposed by Zagier \cite{Zagier}\cite{JacquetZagier1987} wherein instead of integrating the cuspidal kernel directly, one integrates it against a suitable Eisenstein series:
$$I(s) = \int_{G_F Z_{\A} \backslash G(\A)} K_0(x,x) E(x,s) dx$$
The idea of integrating a $G_F Z_{\A}$-invariant function against an Eisenstein series goes back to Rankin and Selberg, who observed that in the region of absolute convergence of the Eisenstein series, this integral equals the Mellin transform of the constant term in the Fourier expansion of the function. Since the residue of $E(x,s)$ at $s=1$ is a constant function, up to the normalizations of the Eisenstein series, we expect to recover the Selberg trace formula by computing $$\Res_{s=1} I(s)$$
This method replaces the truncation procedure with a continuous deformation and gives a more conceptual explanation of the appearance of various terms in the trace formula. Most importantly, class numbers that appear as volumes of torii arise naturally as residues of zeta functions. This approach was initiated by Zagier and Jacquet \cite{JacquetZagier1987} in the 1980's, but the recovery of the Selberg trace formula was completed more recently in 2019 by Han Wu \cite{Wu_2019}.\\
We begin the study of the beyond endoscopic trace formula for $\GL_2$ associated to an arbitrary symmetric power representation $\sigma_k$ of the dual $\GL_2(\C)$. For convenience, we work over $\Q$ and choose $S = \{\infty\}$. In this setting, the relevant trace formula is expressed in terms of the cuspidal kernel $K_{0,s_B}^{k}$ attached to the matrix coefficient for weight $\kappa$ discrete series representation $\phi_{\infty}$ and a special family of basic functions $\phi_{s_B,q}^{k}$ (see section \ref{basicfundefn}) built directly from the $L$-functions $L^S(s, \pi, \rho)$. The novelty lies in treating the problem using Zagier's approach. This helps visualize the trace formula through a continuous deformation lens. More precisely, establishing this trace formula is equivalent to establishing the meromorphic continuation of in $s_B$ variable of $$\Res_{s=1} I(s_B,s) = \Res_{s=1}\int_{\GL_2(\Q)Z(\A) \backslash \GL_2(\A)} K_{0,s_B}^{k}(x,x) E(x,s) dx$$
Integrating against the Eisenstein series transforms the geometric side of the integral $I(s_B,s)$ from the $\GL_2(\Q)$-conjugacy classes into $P(\Q)$-conjugacy classes of $\GL_2(\Q)$. ($P$ denotes the Borel subgroup of upper triangular matrices) to get (refer to \ref{decompgeometric} for details)
\begin{equation}
    I(s_B,s) = I_{\text{ell}}(s_B,s) + I_{\text{hyp}}(s_B,s) + I_{\text{unip}}(s_B,s) + I_{\infty}(s_B,s) \label{Decomp}
\end{equation}
where $I_{\infty}(s_B,s)$ incorporates the special and continuous kernels, along with orbital integrals for elements in $P(\Q)$. 
\\
In this thesis, we start the analysis of the elliptic part
$$T_{\text{ell}}^{k}(s_B) = \Res_{s=1} I_{\text{ell}}(s_B,s)$$
of this trace formula. We want to emphasize that the volumes of torii arise naturally as residues of zeta functions (see \ref{globalorbint}). The elliptic part involves Archimedean orbital integrals weighted by the zeta functions (see \ref{Tellint}). These Archimedean orbital integrals have singularities at central elements (see \ref{archorbint}). Note that, in general, non-Archimedean integrals will also have such singularities, but because we are using basic functions at all finite places, we only encounter singularities from the Archimedean place. These singularities present a major analytical obstacle in the analysis. As in Ali Altug's work, we apply Poisson summation to the trace variable, but the difference is that instead of using an approximate functional equation for the volumes, we use a functional equation for the theta series corresponding to the zeta functions mentioned above. The presence of rapidly decreasing exponential factors in the theta series smooths out the singularities of the Archimedean orbital integrals. This, along with the functional equation makes Poisson summation and further growth analysis possible.
\\
We analyze the dominant term after applying Poisson summation and show the following result.
\begin{theorem}
    The dominant term of the elliptic contribution $$T_{\text{ell}}^{k,\xi = 0}(s_B)$$ admits meromorphic continuation to $\mathfrak{R}(s_B) \geq 0$ with a pole of order $k$ at $s_B = 1$.
\end{theorem}
Here, $\xi$ denotes the Poisson dual variable to the trace variable. Specializing to $k=1$ above recovers Ali Altug's analysis for the standard representation case. (See Section 4.1.2 of \cite{altug2015endoscopytraceformula}). Note that we are considering full level holomorphic cusp forms here, and their symmetric power $L$-functions have been shown to be holomorphic (\cite{newtonThorneI}, \cite{newtonThorneII}). Hence, we expect cancellation of this pole of order $k$ with unipotent and hyperbolic contributions. This cancellation has been shown for the standard representation case by Ali Altug in his work (See Section 2 of \cite{altug2015endoscopytraceformula}). \\
We would like to remark that the magic of Zagier's approach is that the local orbital integrals that we get can be represented naturally via a uniform expression of the form (see \ref{Uniformexp})
$$I^k(\lambda) = \sum_{r=0}^{\infty} C_r^k \wt(r)$$
with $C_r^k$ depending on the symmetric power through evaluation of the basic functions at certain points, and $\wt(r)$ being independent of $k$ (see \ref{eq:wts}). Moreover, we note that for higher symmetric powers, the complications arise from the plethysm coefficients $\mathfrak{c}(k, m, r)$ defined by $$\Sym^m(\sigma_k) = \bigoplus_{r=0}^{\lfloor km/2 \rfloor} \mathfrak{c}(k,m,r) \left( \sigma_{km-2r} \otimes (\det)^r \right)$$
These appear in the calculation of basic functions via Satake transforms (see \ref{kbasicfn}). We exploit the explicit description of these coefficients given by the Cayley-Sylvester formula (See \ref{Cayleyformula}). Although having such explicit formula is very rare, Casselman, through computer experiments has expressed hope in the existence of a completely explicit formula for the symmetric power decompositions of finite-dimensional irreducible representations for complex groups. \cite{casselman2017symmetric}. Such a formula would significantly accelerate the development of the trace formula for other groups. 
\\
We believe that this method, along with inspirations from previous works has significant potential for further development and generalization. Questions of immediate interest include the study of the remaining terms in the trace formula \eqref{Decomp} and the extension to number fields. Another interesting direction would be to deduce functorial transfers using this approach. In particular, it would be interesting to develop the corresponding trace formula for \(L\)-functions for \(\rho = \Sym^2\) arising from lifts of characters of quadratic extensions at some ramified place.
\\
For the convenience of the reader, we outline the roadmap here. We start with a brief overview of beyond endoscopy and basic functions and their explicit calculation for $\GL_2$. We then apply Zagier's approach to the beyond endoscopy trace formula. This is followed by the development of the Archimedean and Non-Archimedean geometric side orbital integrals. We establish the functional equation for the theta series corresponding to Zagier L-functions that appear through the global orbital integral calculations later. Finally, we set up the elliptic part of the trace formula, argue Poisson summation in the trace variable and conclude with the analysis of the dominant term.

\subsection{Acknowledgements}
I am grateful to my advisor, Ng\^o Bao Ch\^au, for his generous mentorship and immense patience throughout this project. I would also like to thank Zhilin Luo for numerous helpful discussions during its development. I am also grateful to Yiannis Sakellaridis for his encouragement throughout my doctoral studies, which culminated in the present work.

\subsection{Notations and Conventions}
\subsubsection*{Notations}
\begin{itemize}
    \item[-] \(\Z, \R, \C\) denote the usual set of integers, real numbers and complex numbers respectively. \(\N= \{0,1,2,\ldots\}\) denotes the set of natural numbers. A subscript denotes the elements of the set satisfying the condition given in the subscript. For example, \(\R_{>0}\) denotes the set of positive reals.
    \item[-] \(\A\) denotes the set of adeles over \(\Q\).
    \item[-] \(\mathcal{S}(\R) = \{ \Phi \in C^{\infty}(\R) : \sup |x^{i}\Phi^{(j)}(x)| < \infty \: \forall \: i,j \in \N\}\) denotes the space of Schwartz function on \(\R\). Similarly, \(\mathcal{S}(\A)\) will denote the Schwartz space of \(\A\). 
    \item[-] \(v_q(x)\) for \(x \in \Q_q\) denotes the \(q\)-valuation of \(x\).
    \item[-] \(\mu\) denotes the Mobius function given by 
    \[
    \mu(n) =\begin{cases}
        1 \qquad & n=1\\
        -1 & n \text{ is product of \(k\) distinct primes}\\
        0 & n \text{ is divisible by a square \(>1\)}
    \end{cases}
    \]
    \item[-] The greatest common divisor of two numbers \(a\) and \(b\) is denoted by \((a,b)\).
    \item[-] The Fourier transform of a function \(F\) will be denoted by \(\widehat{F}\).
    \item[-] The Mellin transform of a function \(\Phi\) is denoted by \(\widetilde{\Phi}\).
    \item[-] For a fundamental discriminant \(D\), \(\delta_D\) is defined by \(\chi_D(-1) = (-1)^{\delta_D}\). Similarly, depending on the context, \(\delta_t\) is defined by 
    \[
    \delta_t = \begin{cases}
        1 \qquad & t^2-4d <0\\
        0 & t^2-4d\geq0
    \end{cases}
    \]
    \item[-] \(\Gamma(z)\) denotes the Gamma function, \(B(\alpha,\beta)\) the Beta function and \(e(x) = e^{2 \pi i x}\). 
    \item[-] \(\int_{(c)}f(u)du\) denotes contour integration along the vertical line \(\mathfrak{R}(u) = c\).
\end{itemize}
\subsubsection*{Conventions}
\begin{itemize}
    \item[-] Hypergeometric branch: \({}_2F_1(a,b;c;z)\) is taken as the standard principal analytic continuation from \(|z| < 1\), with the usual branch cut along \([1,\infty)\) in the \(z\)-plane. On the branch, the hypergeometric function is defined using boundary values. As such, there are 2 ways to to reach the boundary, namely the upper plane and the lower plane. On \(|z| < 1\), the hypergeometric function is defined as follows:
    \[{}_2F_1(a,b;c;z) := \sum_{n=0}^{\infty} \frac{(a)_n(b)_n}{(c)_n} \frac{z^n}{n!}, \qquad (q)_n = \frac{\Gamma(q+n)}{\Gamma(q)}\]
    \((q)_n\) is called the Pochhammer symbol.
    \item[-] Arg(z) and Log(z): For any \(z \in \mathbb{C} \backslash (-\infty, 0]\), define \[\arg(z) \in (-\pi, \pi), \qquad \qquad \log(z) := \ln |z| + i \arg (z)\] Then for any complex exponent \(\alpha\), \(z^{\alpha} := e^{\alpha \log(z)}\).
    \item[-] Principal square root: For \(w \in \mathbb{C} \backslash (-\infty, 0]\), \(\sqrt{w} := e^{\frac{1}{2} \log(w)}\).
    \item[-] Normalization of measures: We use the usual Lebesgue measure on \(\R\). For a local field \(\Q_q\), the additive Haar measure \(dy\) is normalized by \(\vol(\Z_q, dy) =1\) and the multiplicative Haar measure on \(\mathbb{Q}_q^\times\) is given by \( d^\times y = \frac{dy}{|y|_q}\). If \(E/\mathbb{Q}_q\) is a quadratic extension and \(\mathcal{O}_E\) is the ring of integers of \(E\), then the measure \(d^\times t\) on \(E^\times\) is normalized so that \(\vol(\mathcal{O}_E^\times, d^\times t) = 1\).
\end{itemize}

\section{Overview of Beyond Endoscopy}
We work over $\Q$ for notational simplicity and convenience. 
\\
Langlands proposal of beyond endoscopy proposes a trace-formula approach to functoriality with an aim of detecting automorphic transfer through analytic behavior of automorphic $L$-functions. Its guiding principle is that the poles of these $L$-functions encode information about how automorphic representations of a given group arise from smaller groups through functoriality. We give a brief explanation of this below and develop the stable trace formula in which automorphic $L$-functions appear directly on the spectral side. For detailed study, refer to \cite{Langlands2004} \cite{Tagaki} \cite{frenkel2010formuledestraceset}. 
\\
Let $G$ be a reductive group over a global field $\Q$, and let 
\[
\rho : \prescript{L}{}{G} \rightarrow \GL(V_{\rho})
\]
be a finite-dimensional representation of the $L$-group of $G$. For an automorphic representation $\pi = \otimes_{q} \pi_{q}$ of $G(\A_{\Q})$, consider the partial $L$-function 
\[
L^S(s_B,\pi,\rho) = \prod_{q \not \in S}L_{q}(s_B, \pi_q, \rho)
\]
where $S$ is a finite set of places containing all the ramified and archimedean places and the local $L$-factors are defined by 
\begin{equation}
    L_q(s_B, \pi_q, \rho) = \det(1- \rho(\alpha_{\pi_q})q^{-s_B})^{-1} \label{eq: lfundef}
\end{equation}
where $\alpha_{\pi_q} \in \prescript{L}{}{G}$ is the local parameter of $\pi_q$.
According to the guiding conjectures, the behavior of this $L$-function near the line $\mathfrak{R}(s_B) = 1$ is governed by the way $\pi$ arises from smaller groups by functorial transfer. In particular, poles on the line $\mathfrak{R}(s_B) =1$ are expected to signal that $\pi$ comes from a smaller subgroup on the dual side. 
\\
More precisely, an automorphic representation $\pi = \pi_G$ determines, in addition to its local semisimple parameters, a homomorphism 
\[
\eta : \SL_2(\C) \rightarrow \prescript{L}{}{G}
\]
together with a subgroup $\lambda_H \subset \prescript{L}{}{G}$ centralizing the image of $\eta$, and finally a reductive group $H$ together with an $L$-homomorphism 
\[
\psi : \prescript{L}{}{H} \rightarrow \lambda_H \subset \prescript{L}{}{G}
\]
such that $\pi$ is related to an automorphic representation $\pi_H$ of $H$ by functorial transfer. If we restrict $\rho$ along $\eta \times \psi$, we obtain a decomposition 
\[
\rho \circ (\eta \times \psi) \cong \bigoplus_{j} \sigma_j \otimes \rho_H^j
\]
with $\sigma_j$ denoting irreducible representations of $\SL_2(\C)$. This leads to the formal factorization 
\begin{equation}
    L^S(s_B, \pi, \rho) = \prod_{j} \prod_{i} L^{S}(s_B+i, \pi_H, \rho_H^j) \label{eq:Lfactorization}
\end{equation}
where the shifts $i$ are determined by the weights of $\sigma_j$. This formula shows that the analytic behavior of $L^S(s_B, \pi, \rho)$, especially its poles, is governed by the smaller group $H$ together with the nontrivial $\SL_2$-factor $\eta$. In particular, nontrivial $\eta$ shifts possible poles to the right of the central line and makes them, at least in theory, detectable by analytic means. 
\\
Thus, poles of automorphic $L$-functions are spectral markers of functorial transfer. If $\pi$ is transferred from a subgroup $H$, and if the restriction of $\rho$ to $\prescript{L}{}{H}$ contains the trivial representation, then we expect $L^S(s_B, \pi,\rho)$ to acquire a pole at $s_B=1$. More generally, the order of the pole measures the multiplicity of the trivial representation in the restriction of $\rho$ to the relevant subgroup of the dual group. This way, the pole order is a representation-theoretic invariant that is expected to distinguish different functorial sources inside the automorphic spectrum.
\\
We now move to establishing how the trace formula setup can be used to extract analytical information about $L$-functions. Langlands original proposal was to set up the trace formula by considering logarithmic derivatives of $L$-functions, since residues of logarithmic derivatives are additive and hence well suited to separate contributions according to pole order. As mentioned in the introduction, this leads to studying arithmetic sums remniscent of those appearing in the proof of the prime number theorem. We develop the trace formula below following Sarnak's suggestion to get a sum over integers that is more amenable to adelic harmonic analysis.
\\
To make this precise, we define local test functions $\phi_{s_B,\nu}^{\rho}$ on $G(\Q_{\nu})$ that are smooth compactly supported for $\nu \in S$, and unramified outside $S$. For $\nu \not \in S$, one introduces the basic function $\phi_{s_B,\nu}^{\rho}$ that satisfies 
\[
\tr\left(\pi_{\nu}(\phi_{s_B,\nu}^{\rho})\right) = \begin{cases}
    L_{\nu}(s_B, \pi_{\nu}, \rho) \qquad \qquad &\pi_{\nu} \text{ is unramified}\\
    0 & \text{otherwise}
\end{cases} 
\]
Consequently, if we define the global test function 
\[
\phi_{s_B} = \otimes_{\nu} \: \phi_{s_B, \nu}^{\rho}
\]
then the corresponding stable trace formula has spectral side given by 
\[
\sum_{\pi^{\text{st}}} m(\pi^{\text{st}}) L^S(s_B, \pi^{\text{st}}, \rho) \prod_{\nu \in S} \tr \: \pi^{\text{st}}_{\nu}(\phi_{s_B, \nu}^{\rho})
\]
The automorphic $L$-functions now appear directly on the spectral side, and the analytic behavior in the variable $s_B$ is built into the kernel itself.
\\
The stable trace formula is expected to be used recursively. First, one can isolate the part of the spectrum corresponding to non-tempered parameters (those with nontrivial $\eta$). Since the associated $L$-functions have poles shifted to the right of $\mathfrak{R}(s_B) =1$ from the integers coming from the $\SL_2$-factor \ref{eq:Lfactorization}, we hope to remove such contributions recursively, group by group. After this, one is left with those representations for which $\eta$ is trivial, namely the representations that should satisfy the Ramanujan conjecture. Among these representations, one seeks to isolate those for which $L^S(s_B, \pi, \rho)$ has a pole of prescribed order at $s_B  =1$. These are precisely the representations coming by transfer from those subgroups $H \subset \prescript{L}{}{G}$ on which the restriction of $\rho$ contains a trivial summand.
\\
Frenkel-Langlands-Ng\^{o} \cite{frenkel2010formuledestraceset} emphasized that after stabilization, and after careful treatment of the measure factors, the elliptic regular part of the geometric side of the trace formula is almost a sum over a finite-dimensional vector space over $\Q$. The authors used Getz truncation to overcome the analytical difficulties posed by the singularities of orbital integrals at ramified places, and isolated the contribution of the trivial representation in the dominant term for a general group $G$ that is semi-simple, simply connected and satisfies $G = G_{\text{der}}$ (the derived group of $G$). Ali Altug \cite{Altu__2015} rediscovered the result in the case of $\GL_2$, where instead of truncation, he used an approximate functional equation that allowed him to get an expression suitable for further analysis. \cite{altug2015endoscopytraceformulaII}\cite{altug2015endoscopytraceformula}.
\\
Our approach is similar to these methods in its use of Poisson summation, and is particularly close to the approach of Ali Altug in that it also relies on a functional equation. The principal difference is that we employ Zagier's method to obtain a form in which the trace formula emerges through residues, at the cost of introducing an additional complex variable $s$. We expect this perspective to provide a clearer understanding of the constituents of the trace formula and hence, to be more generalizable. In particular, Zagier's approach naturally produces zeta functions incorporating the volume terms appearing in the trace formula, and directly yields the local representations for them that play a central role in establishing the functional equation of these zeta functions.

\section{Basic functions}
One of the central goals of the Langlands program is the study of the analytic properties  of automorphic $L$-functions, particularly their meromorphic continuation and functional equations. One approach to these questions is through functoriality, as mentioned in the introduction. Yet another approach, due to Roger Godement and Hervé Jacquet \cite{godement2006zeta}, generalizes the earlier work of John Tate on the Riemann zeta function $\zeta(s)$ \cite{Tate1950}. In their work, they established the meromorphic continuation and functional equation of the $L$-function $L(s_B, \pi, \rho)$, associated to an automorphic representation $\pi$ of $\GL_n(\Q) \backslash \GL_n(\A)$ and the standard representation $\rho$ of the dual group. Their method relies on Fourier analysis on the Schwartz space of the matrix algebra $M_n$, together with a global Poisson summation formula.
\\
This led to the question of how to adapt this approach of Godement-Jacquet to other groups. In particular, what will replace the matrix algebra and the Schwartz functions? Braverman and Kazhdan made a series of conjectures aimed at answering these questions for general  automorphic $L$-functions \cite{BK}.The program conjectures the existence of local Schwartz spaces $\mathcal{S}^{\rho}(G(\Q_{\nu}))$, existence of $\rho$-basic functions within this space, construction of global Schwartz space, an involutive $\rho$-Fourier transform and a Poisson summation formula for the global $\rho$-fourier transform. For a modern overview of the Braverman-Kazhdan program, we refer the reader to \cite{Tagaki}. These conjectures were known for torii for quite sometime \cite{Tagaki}, and more recently, the $\rho$-Fourier transform was established and made explicit in the recent paper of Luo and Ng\^{o} \cite{luo2024nonabelianfourierkernelsmathrmsl2} in full generality in the case of $\SL_2$ and $\GL_2$. 
\\
Beyond endoscopy suggests using the trace formula for these $\rho$-basic functions to study the $L$-functions. That is, these basic functions serve as test functions in our trace formula development ahead. We begin by recalling the definition of the $\rho$-basic function and calculate it explicitly for $\GL_2$ in this section. 

\begin{definition}
\label{basicfundefn}
    The conjectured space of $\rho$-Schwartz functions $\mathcal{S}^{\rho}(G(\Q_{\nu}))$ contains a distinguished vector, the $\rho$-basic function $\phi_{s_B,\nu}^{\rho}
    $ that satisfies 
    \[
    \tr\left(\pi_{\nu}(\phi_{s_B,\nu}^{\rho})\right) = \begin{cases}
        L_{\nu}(s_B, \pi_{\nu}, \rho) \qquad \qquad &\pi_{\nu} \text{ is unramified}\\
        0 & \text{otherwise}
    \end{cases} 
    \] \label{defbasicfn}
\end{definition}
From the definition of the local $L$-function \eqref{eq: lfundef}, it follows that computing basic function has two major components: Satake transform and their inverses, and symmetric power decompositions of $\rho$. Both these components can be easily made explicit for $\GL_2$ case, which we do next. We refer to Casselman's work \cite{casselman2017symmetric} for modern proofs and a more involved study of these components. 

\subsection{Basic functions for $\GL_2$}
We now compute the unramified local basic function at a finite place $q$ attached to the $k$-th symmetric power representation $$\rho = \sigma_k = \Sym^k(W)$$ where $W$ is the standard two-dimensional representation of $\GL_2(\C)$. For this subsection, let 
\[
G = \GL_2(\Q_q), \qquad K = \GL_2(\Z_q) 
\]
Let $A$ denote the subgroup of diagonal matrices in $G$ and $\mathcal{H} = C_c^{\infty}(K \backslash G/ K)$ be the spherical Hecke algebra of compactly supported \(K\)-bi-invariant functions on \(G\).

\subsubsection{Satake transform for $\GL_2$}
The unramified characters of $A$ are parametrized by $\alpha, \beta \in \C^{\times}$ via:
\[
\chi_{[\alpha,\beta]} : \left( \begin{matrix}
    q^{m} & 0 \\ 0 & q^n
\end{matrix}\right) \mapsto \alpha^m \beta^n
\]
If $\pi_{[\alpha,\beta]}$ denotes the unramified principal series representation associated with $\chi_{[\alpha,\beta]}$, then the subspace of $K$-fixed vectors $\pi_{[\alpha,\beta]}^K$ is one dimensional and any function in $\mathcal{H}$ acts on this space by a scalar. This gives the Satake homomorphism 
\begin{align}
    \mathfrak{S} : \mathcal{H} &\rightarrow \C[\alpha^{\pm1}, \beta^{\pm1}]^{S_2} \\
    \phi &\mapsto \mathfrak{S}(\phi) \: \left( \text{such that } \tr_{\pi}(\phi) = \mathfrak{S}(\phi)(\alpha_\pi) \right) \label{eq:satakedef}
\end{align}
from the Hecke algebra to the space of symmetric polynomials in $\alpha^{\pm1}, \beta^{\pm1}$. The following lemma gives the explicit description of the Satake transform without proof for $\GL_2$. 
\begin{lemma}
    Consider the normalized basis 
    \[
    \left\{f_{m,n} = q^{-(m-n)/2} \mathbf{1}_{K\left( \begin{matrix}
        q^m & \\ & q^n
    \end{matrix}\right)K}\right\}_{m \geq n}
    \]
    of $\mathcal{H}$ and the basis of $\C[\alpha^{\pm1},\beta^{\pm1}]^{S_2}$ given by 
    \[
    \left\{\tau_{m+n,n}(\alpha, \beta) = \alpha^n \beta^n \left( \alpha^m + \alpha^{m-1}\beta + \cdots + \beta^m \right)\right\}_{m, n \in \Z_{\geq 0}}
    \]
    In terms of these basis, the Satake transform $\mathfrak{S}$ is given by 
    \begin{equation}
        \mathfrak{S}(f_{m,n}) = \begin{cases}
            \tau_{m,n} \qquad \qquad & m = n,n+1\\
            \tau_{m,n} - q^{-1} \tau_{m-1,n+1} & m \geq n+2
        \end{cases} \label{eq: satake}
    \end{equation}
\end{lemma}
Since this matrix is triangular and unipotent, the corollary below follows:
\begin{corollary}
    For $\GL_2$, the Satake inverse $\mathfrak{S}^{-1}$ takes the following form
    \begin{equation}
        \mathfrak{S}^{-1}(\tau_{m,n}) = \sum_{\substack{m \geq j \geq i \geq n \\j+i = m+n }} q^{n-i} f_{j,i} \label{eq:satinv}
    \end{equation}
\end{corollary}
We will use this form of the Satake inverse to calculate the $\rho$-basic function below. Before that, we discuss the second component: the symmetric power decompositions for $\GL_2$. 

\subsubsection{Decomposition of $\Sym^m(\sigma_k)$}
The decomposition theorem for symmetric powers is well known. We recall the result here.
\begin{lemma} \label{Cayleyformula}
    For $\GL_2$, the representation $\Sym^m(\sigma_k)$ decomposes into irreducible as 
    \begin{equation}
        \Sym^{m}(\sigma_k) = \bigoplus_{r=0}^{km/2} \mathfrak{c}(k,m,r) \left(\Sym^{km-2r}(W) \otimes (\det)^r \right) \label{eq: symdecomp}
    \end{equation}
    where the Plethysm coefficients $\mathfrak{c}(k,m,r)$ denote the multiplicities in the decomposition. These are given by the Cayley-Sylvester formula: 
    \begin{equation}
        \mathfrak{c}(k,m,r) = p_{m.k}(r) - p_{m,k}(r-1), \qquad p_{m,k}(-1) = 0, \qquad 0 \leq r \leq \lfloor \frac{km}{2} \rfloor \label{eq:coeff}
    \end{equation}
    with $p_{m,k}(r)$ defined by 
    \begin{equation}
        \sum_{r=0}^{km} p_{m,k}(r) q^{r} = \prod_{j=1}^{k} \frac{1-q^{m+j}}{1-q^j} = \qbinom{m+k}{k}_q \label{eq:qbinom}
    \end{equation}
\end{lemma}
\begin{remark}
    \begin{enumerate}
        \item $p_{m,k}(r)$, the coefficients of the $q$-Gaussian binomial \eqref{eq:qbinom} give the number of partitions of $r$ with $k$ or fewer parts each less than or equal to $m$. 
        \item We emphasize that the telescoping nature \eqref{eq:coeff} plays an important role in making the analysis of the dominant feasible (see \ref{Dw})
    \end{enumerate}
\end{remark}
This formula was first given by Cauchy \cite{cauchy1843memoire} in 1843. For decomposition results of symmetric powers for arbitrary reductive groups, see \cite{Sturmfels1995Vector} \cite{SymCaselmann}. Although no comprehensive generalization of such simple formula for the plethysm coefficients is known, computer experiments hint at their existence for other groups \cite{casselman2017symmetric}. 
\\
We now put together the Satake transform and the decomposition above to calculate the basic functions for $\sigma_k$. 
\begin{proposition} \label{kbasicfn}
    The basic function $\phi_{s_B,q}^k$ corresponding to the $k$-th symmetric power representation $\sigma_k$ of $\GL_2(\C)$ is given by
    \begin{equation}
        \phi_{s_B,q}^{k} = \sum_{m=0}^{\infty} \left( \sum_{r=0}^{\lfloor km/2 \rfloor} \mathfrak{c}(k,m,r) q^r T_{k,q}^{m,r} \right) q^{-m(s_B+\frac{k}{2})} 
    \end{equation} 
    with $T_{k,q}^{m,r}$ being the characteristic function of the set
    \[
    \left\{ X \in \text{M}_2(q^r \Z_q) \: : \: |\det X|_q = q^{-km} \right\}
    \] \label{Propbasicfn}
\end{proposition}
\begin{proof}
    From the definitions of the basic function \ref{defbasicfn} and Satake transform \ref{eq:satakedef}, it is clear that the basic function is defined by 
    \begin{align*}
        \mathfrak{S}(\phi_{s_B,q}^{k})\:(\gamma) = \det(1-\sigma_k(\gamma) q^{-s_B})^{-1}
    \end{align*}
    for a semisimple Satake parameter $\gamma = \left( \begin{matrix}
        \alpha & \\ & \beta
    \end{matrix}\right) \in \GL_2(\C)$. Expanding the right hand side using Molien's formula and using the decomposition of symmetric powers, we get 
    \begin{align*}
        \mathfrak{S}(\phi_{s_B,q}^{k})\:(\gamma) &= \sum_{m=0}^{\infty} \tr\left( \Sym^m\sigma_k(\gamma)\right) q^{-ms_B} \\
        &= \sum_{m=0}^{\infty} \left( \sum_{r=0}^{\lfloor km/2 \rfloor} \mathfrak{c}(k,m,r) \tr \left( \Sym^{km-2r}(W)(\gamma)\right)\right) \tr \left( (\det(\gamma))^r\right) q^{-ms_B} \\
        &= \sum_{m=0}^{\infty} \left( \sum_{r=0}^{\lfloor km/2 \rfloor} \mathfrak{c}(k,m,r) \alpha^r \beta^r \left( \sum_{r=0}^{\lfloor km/2 \rfloor} \alpha^{km-2r-j} \beta^j\right)\right) q^{-ms_B} \\
        &= \sum_{m=0}^{\infty} \left( \sum_{r=0}^{\lfloor km/2 \rfloor} \mathfrak{c}(k,m,r) \tau_{km-r,r}(\alpha,\beta)\right) q^{-ms_B}
    \end{align*}
    Note that $\mathfrak{S}(\phi_{s_B,q}^{k})\:(\gamma)$ can be easily verified to be absolutely convergent for $\mathfrak{R}(s_B) \gg 0$. Now, the result follows by applying Satake inverse \eqref{eq:satinv} on both sides and groups the terms by determinant valuations.
\end{proof}
\begin{remark}
    The appearance of the shift by $k/2$ is exactly what one expects from the general theory. In our normalization, the basic function associated to a highest weight $\lambda$ is naturally attached to the shifted $L$-factor $L(s_B - \langle \lambda, \rho^{\vee} \rangle, \pi, \sigma_{\lambda})$. For $\GL_2$ and $\lambda = k \varepsilon_1$, we have 
    \[
    \langle \lambda, \rho^{\vee} \rangle = \frac{k}{2}
    \]
    The factor $q^{-km/2}$ appearing above is just the local manifestation of this shift.
\end{remark}
We conclude this section by verifying that the local basic functions are integrable for $\mathfrak{R}(s_B) \gg 0$. Note that the bound given below is in no way optimal. 
\begin{lemma}
    For $\mathfrak{R}(s_B) \geq \frac{k}{2}$, $\phi_{s_B,q}^{k} \in L^{1}(\GL_2(\Q_q))$ \label{integrabilitybasicfn}
\end{lemma}
\begin{proof}
    It is enough to show that 
    \[
    || \phi_{s_B, q}^k ||_1 \leq  \sum_{m=0}^{\infty} \sum_{r=0}^{\lfloor \frac{km}{2} \rfloor} |\mathfrak{c}(k,m,r)| q^{r} q^{-m(\mathfrak{R}(s_B) + \frac{k}{2})} \vol(T^{m,r}_{k,q}) < \infty
    \]
    Because the Haar measure on \(G\) is multiplication invariant, it follows that 
    \[
    \vol(T^{m,r}_{k,q}) = \vol \{Y \in M_2(\Z_q)\: : \: v_q(\det Y) = km-2r\}
    \]
    Using the Cartan double cosets volume given by 
    \[
    \vol\left( K \left( \begin{smallmatrix}
        q^{l} & \\ & 1
    \end{smallmatrix} \right) K \right) = \begin{cases}
        1 \qquad & l = 0\\
        q^l + q^{l-1} & l > 0
    \end{cases}
    \]
    we get 
    \[
    \vol(T^{m,r}_{k,q}) \ll \sum_{b=0}^{\lfloor \frac{n}{2} \rfloor} q^{km-2r-2b} \leq q^{km-2r}
    \]
    and hence, 
    \[
    || \phi_{s_B,q}^k ||_1 \leq \sum_{m=0}^{\infty} \left( \sum_{r=0}^{\lfloor \frac{km}{2} \rfloor} |\mathfrak{c}(k,m,r)| \right) q^{-m(s_B-\frac{k}{2})}
    \]
    The \(r\)-sum has at most polynomial growth, and thus, when \(\mathfrak{R}(s_B) > \frac{k}{2}\), the exponential decay beats the polynomial growth of the \(r\)-sum and the sum converges.
\end{proof}

\section{Zagier treatment of the beyond endoscopy trace formula}
\subsection{Choice of test functions and conventions} \label{testfnchoice}
From now on, we fix the following notations and test functions, unless otherwise stated. 
\begin{itemize}
    \item[-] $G = \GL_2$, $G_{\A} = G(\A)$, $Z$ denotes the center of $G$, $Z_{\infty}^{+} \subset Z(\R)$ is the identity component of the archimedean center. We work over the field of rationals $\Q$.
    \item[-] $P \subset G$ is the standard upper triangular Borel subgroup. 
    \item[-] We fix a decomposition $$\A^{\times} = \A^{\times}_1 \times \R^{\times}_{>0}$$
    where $\A^{\times}_1$ is the subgroups of ideles $\A^{\times}$ of norm 1. 
    \item[-] At the archimedean place, we choose the test function for the trace formula $\phi_{\infty}$ to be the matrix coefficient of discrete series representation of even weight $2\kappa$, which is given by 
    \begin{equation}
        \phi_{\infty}(g) = \begin{cases}
            \frac{2\kappa-1}{4\pi} \frac{(\det g)^{\kappa} (2i)^{2\kappa}}{(-b+c+(a+d)i)^{2\kappa}} \qquad \qquad & g = \left(\begin{smallmatrix}
                a &b \\ c& d
            \end{smallmatrix} \right), \det g >0\\
            0 & \det g < 0
        \end{cases} \label{eq:archphi}
    \end{equation}
    \item[-] At all the finite places, we set the test function to be the basic function for $\sigma_k$, which as shown in the last section, is given by 
    \begin{equation}
        \phi_{s_B,q}^{k} = \sum_{m=0}^{\infty} \left( \sum_{r=0}^{\lfloor km/2 \rfloor} \mathfrak{c}(k,m,r) q^r T_{k,q}^{m,r} \right) q^{-m(s_B+\frac{k}{2})} \label{eq:basicfn}
    \end{equation} 
    with $T_{k,q}^{m,r}$ as given in \ref{Propbasicfn}. 
    \item[-] We specialize $\Phi \in \mathcal{S}(\A^2)$ to $\Phi = \Phi_{\infty} \times \prod_{q} \Phi_{q}$ given by 
    \begin{equation}
        \Phi_{\infty} (x,y) = e^{-(x^2+y^2)}, \qquad \Phi_q(x,y) = \mathbf{1}_{\Z_q} \times \mathbf{1}_{\Z_q} \label{Phi}
    \end{equation} 
\end{itemize}

\subsection{Zagier deformation of the trace formula}
Let $\phi_{s_B} \in C^{\infty}(Z_{\infty}^{+} \backslash G_{\A})$ with $\phi_{s_B} = \phi_{\infty} \times \prod_{q} \phi_{s_B,q}^k$ and let $\Pi$ be the representation of $G_{\A}$ on $L^2(X)$, $X =Z_{\infty}^{+}G(\Q) \backslash G(A)$ given by right translation, and $\Pi_0$ be its restriction to the cusp forms $L^2_{\cusp}(X)$. We have the following spectral decomposition of $\Pi$ given by \begin{equation}
    L^2(X) = L^2_{\cusp}(X) \oplus L^{2}_{\text{res}}(X) \oplus L^2_{\cont}(X)
\end{equation}
where $L^2_{\cusp}(X)$ is the cuspidal spectrum, $L^2_{\text{res}}(X)$ is the residual spectrum spanned by one-dimensional representations of the form $\chi\circ\det$, where $\chi$ runs through continuous characters of $\Q^{\times} \backslash \A^{1}$ and $L^2_{\text{cont}}(X)$ is the continuous spectrum generated by Eisenstein series. 
\\
To $\phi_{s_B}$, we associate the automorphic kernel 
\begin{equation}
    K_{s_B}(x,y) = \sum_{\gamma \in G(\Q)} \phi_{s_B}(x^{-1} \gamma y)
\end{equation} 
which is well defined on $X \times X$ thanks to \ref{integrabilitybasicfn}. This kernel decomposes according to the spectral decomposition of $L^2(X)$ as 
\begin{equation}
     K_{s_B}(x,y) =  K_{s_B, \cusp}(x,y) +  K_{s_B, \text{res}}(x,y)  + K_{s_B,\cont}(x,y) \label{eq:spectraldecomp}
\end{equation}
For $\Phi \in \mathcal{S}(\A^2)$ (which we will soon specialize to \ref{Phi}), the corresponding Eisenstein series is given by \[
E(g,s) = \sum_{\gamma \in P(\Q) \backslash G(\Q)} f_{\Phi}(s, \gamma g)
\]
where the Godement section is given by 
\begin{equation}
    f_{\Phi}(s,g) = |\det g|^s \int_{\A^{\times}} \Phi\left( (0,t)g\right) |t|^{2s} d^{\times} t \label{eq:f}
\end{equation}
For $\mathfrak{R}(s) > 1$, this series is absolutely convergent. By the standard theory of Eisenstein series for $\GL_2$, it admits meromorphic continuation to the whole complex plane and satisfies the functional equation given by 
\[
     E(g,s) = \frac{\pi^{-(1-s)} \Gamma(1-s) \zeta(2-2s)}{\pi^{-s} \Gamma(s) \zeta(2s)} E(g,1-s) 
\]
and has a simple pole only at $s=1$ with residue independent of $g$:
\begin{equation}
    \Res_{s=1} E(g,s) = \frac{3}{\pi} \label{eq:EisRes}
\end{equation}
Following Zagier \cite{Zagier} and Jacquet-Zagier \cite{JacquetZagier1987} approach, we consider the following integral 
\[
I(s_B,s) = \int_{\X} K_{s_B,\cusp}(x,x) E(x,s) dx
\]
We first show that this integral converges absolutely for large $\mathfrak{R}(s_B)$:
\begin{proposition}
    The integral $I(s_B,s)$ \eqref{eq:EisRes} is absolutely convergent for $\mathfrak{R}(s_B) > 1+ \frac{k}{2}$ and for all $s$, with $E(x,s)$ realized by meromorphic continuation outside $\mathfrak{R}(s)> 1$.
\end{proposition}
\begin{proof}
    Recall that the cuspidal kernel can be given by \[K_{s_B,\cusp}(x,x) = \sum_{\pi \subset L^2_{\cusp}} \sum_{\varphi \in \mathcal{B}(\pi)} (R(\phi_{s_B})\varphi)(x) \overline{\varphi(x)} \] where \(\mathcal{B}(\pi)\) is an orthonormal basis and \(R(\phi_{s_B}) = \otimes_{\nu} \pi_{\nu}(\phi_{s_B,\nu})\). As \(\phi_{s_B,\infty}\) is the matrix coefficient for the weight \(\kappa\) discrete series representation \(D_{\kappa}\), \(\pi(\phi_{s_B}) = 0\) unless \(\pi_{\infty} = D_{\kappa}\). Again, by construction of the basic function, \(\pi(\phi_{s_B}) = 0\) unless \(\pi_{q}\) is unramified. Thus, \(R(\phi_{s_B})\) is a finite rank operator on \(L^{2}_{\cusp}\) and we can diagonalize the kernel to write 
    \[
    K_{s_B,\cusp}(x,x) = \sum_{\pi \in \mathcal{A}_k} d_k^{-1} L^{\text{fin}}(s_B,\pi,\sigma_k) |\varphi_{\pi}(x)|^2
    \]
    where \(\mathcal{A}_k = \{\pi_{\varphi} \:: \: \varphi \text{ is k-th level Hecke eigenform}\}\) and \(\varphi_{\pi}\) be the \(L^2\) normalized new vector. Note that this is a finite sum. Now, \[
    L^{\text{fin}}(s_B, \pi, \sigma_k) = \prod_q\prod_{i=0}^{k} (1- \alpha_q^{k-i}\beta_{q}^{i} q^{-s_B})^{-1}   \]
    where \(\text{diag}(\alpha_q,\beta_q)\) denotes the Satake parameter of \(\pi_q\). The standard Fourier-coefficient estimate for Hecke form gives \(|\alpha_q|, |\beta_q| \ll q^{\frac{1}{2}}\). Taking logarithms, this implies 
    \[
    \log|L^{\text{fin}}(s_B,\pi,\sigma_k)| \ll \sum_{q} (k+1)q^{\frac{k}{2} - \mathfrak{R}(s_B)}
    \] which converges for \(\mathfrak{R}(s_B) > 1 +\frac{k}{2}\). Now because \(\varphi_{\pi}\) is a cusp form, its constant term vanishes and \[
    |\varphi_{\pi}(x)|^2 \ll_{N} x^{-N} \qquad \text{for every \(N\)}
    \]
    On the other hand, for \(\mathfrak{R}(s) > 1\), the series defining \(E(x,s)\) converges absolutely and has at most polynomial growth \[
    |E(x,s)| \ll x^{\mathfrak{R}(s)}
    \]
    Thus \(|\varphi_{\pi}(x)|^2 |E(x,s)| \ll_N x^{\mathfrak{R}(s)-N} \quad \forall \:N\) is itself rapidly decreasing and \(\X\) has finite volume and the result follows.
\end{proof}
\begin{remark}
    The bound derived above is not optimal. The key was bounding the local parameters \(\alpha_q, \beta_q\). Kim-Sarnak \cite{Kim-Sarnak} have improved the bound to $1+\frac{7k}{64}$, and the strongest bound of $\mathfrak{R}(s_B) >1$ is given by Deligne \cite{Deligne1971} \cite{Deligne}. We are working with the loosest bound here, because the idea is to extend this region of convergence through analytic techniques without using these advanced geometric results. 
\end{remark}
Thus, we are interested in studying the meromorphic continuation of $I(s_B,s)$ in the variable $s_B$ via the geometric side \ref{decompgeometric}. Then,  due to \eqref{eq:EisRes};
\begin{equation*}
    \Res_{s=1} I(s_B,s)
\end{equation*}
will, in principle recover the analytically continued beyond endoscopy trace formula for $\GL_2$ and $\rho = \sigma_k$.
With this in mind, we start with unfolding the geometric side of the integral.

\subsection{Geometric decomposition of trace formula} \label{decompgeometric}
The diagonal cuspidal kernel can be broken according to conjugacy classes as 
\begin{equation*}
    K_{s_B,\cusp}(x,x) = \sum_{C} \mathcal{K}_{s_B,C}(x) + \mathcal{K}_{s_B,\infty}(x) 
\end{equation*}
where the sum is over the non-central conjugacy classes $C \subset G(\Q)$,  
\begin{equation*}
    \mathcal{K}_{s_B,C}(x) = \sum_{\substack{\lambda \in C \\ \lambda \not \in P(\Q)}} \phi_{s_B}(x^{-1} \lambda x)
\end{equation*}
and 
\begin{equation*}
    \mathcal{K}_{s_B,\infty}(x) = \sum_{\lambda\in P(\Q)}\phi_{s_B}(x^{-1}\lambda x)-K_{s_B,\cont}(x,x)-K_{s_B,\text{res}}(x,x).
\end{equation*}
Because each of these expressions are $P(\Q)$-invariant, for $\mathfrak{R}(s) > 1$, the Eisenstein series unfolds to give 
\begin{align*}
    I(s_B,s) &= \int_{\X} K_{s_B,\cusp}(x,x) \sum_{\gamma \in P(\Q) \backslash G(\Q)} f(\gamma x, s) dx \\
    &= \int_{\XP} K_{s_B,\cusp}(x,x) f(x,s) dx
\end{align*}
Hence, 
\[
I(s_B,s) = \sum_{C} I_C(s_B,s) + I_{\infty}(s_B,s)
\]
where \[I_{\infty}(s_B,s) = \int_{\XP} \mathcal{K}_{s_B,\infty}(x) f(x,s) dx\] and 
\[
I_C(s_B,s) = \int_{\XP} \mathcal{K}_{s_B,C}(x)f(x,s) dx
\]
The non-central conjugacy classes further decompose into elliptic, hyperbolic and unipotent classes to give 
\begin{equation*}
    I(s_B,s) = I_{\Ell}(s_B,s) + I_{\hyp}(s_B,s) + I_{\unip}(s_B,s) + I_{\infty}(s_B,s)
\end{equation*}
with 
\begin{itemize}
    \item[-] Elliptic contribution: Each elliptic element $\lambda \in G(\Q)$ generates a quadratic extension $E = \Q[\lambda]$, and its centralizer is $G(\Q)_{\lambda} = E^{\times}$. Fixing an embedding $E^{\times} \hookrightarrow G(\Q)$ for each quadratic extension $E/\Q$, the elliptic contribution can we written as 
    \begin{align*}
        I_{\Ell}(s_B,s) = \sum_{[E:\Q] = 2} I_{E}(s_B, s), \qquad I_{E}(s_B,s) = \frac{1}{2} \sum_{\substack{\lambda \in E^{\times}\\ \lambda \not \in \Q}} I_{\{\lambda\}}(s_B, s)
    \end{align*}
    \item[-] Hyperbolic contribution $I_{\hyp} = 
     \sum_{C \: \hyp} I_C(s_B,s)$
    \item[-] Similarly, the unipotent contribution $I_{\unip}(s_B,s)=  \sum_{C \: \unip} I_C(s_B,s)$
\end{itemize}
\begin{remark}
    Observe that integrating the kernel against the Eisenstein series reorganizes the kernel expansion from a sum over $G$-conjugacy classes to a sum over $P$-conjugacy classes.
\end{remark}

\subsection{Simplification of elliptic, hyperbolic and unipotent orbital integrals}
For a non-central conjugacy class $C$ represented by $\lambda$, any element of the conjugacy class can be represented by $\gamma^{-1} \lambda \gamma$ for $\gamma$ determined uniquely up to left multiplication by an element of the centralizer $G(\lambda)_{\Q}$; and hence we can write 
\begin{equation}
    \mathcal{K}_{s_B,C}(x) = \sum_{\substack{\gamma \in G(\lambda)_\Q \backslash G(\Q) \\ \gamma^{-1}\lambda\gamma \not \in P(\Q)}} \phi_{s_B}(x^{-1}\gamma^{-1}\lambda \gamma x) \label{eq:KC}
\end{equation}
This set can be represented simply as the following lemma shows
\begin{lemma} \label{conclassdecomp}
    For any non-central conjugacy class $C$ represented by an element $\lambda$, 
    \begin{equation}
        \{\gamma \in G(\lambda)_\Q \backslash G(\Q) \: : \: \gamma^{-1}\lambda\gamma \not \in P(\Q)\} \cong \lambda_0 P(\Q)/Z(\Q)
    \end{equation}
    with 
    \[
    \lambda_0 = \begin{cases}
        (\begin{smallmatrix}
            1 & 0 \\ 0 & 1
        \end{smallmatrix}) \qquad\qquad & \lambda \text{ is elliptic }\\
        (\begin{smallmatrix}
            1 & 0 \\ 1 & 1
        \end{smallmatrix}) \qquad\qquad &\lambda \text{ is hyperbolic}\\
        (\begin{smallmatrix}
            0 & 1 \\ 1 & 0
        \end{smallmatrix}) \qquad\qquad &\lambda \text{ is unipotent}
    \end{cases}
    \]
\end{lemma}
\begin{remark}
    The lemma is essentially a reformulation of the fact that the conjugation action of \(P(\Q)\) on \(G(\Q)\setminus P(\Q)\) is free.
\end{remark}
\begin{proof}
    \begin{itemize}
        \item[-] If $\lambda$ is elliptic, none of its conjugates lie in $P(\Q)$. Moreover,
        \[
        G(\lambda)_\Q P(\Q)=G(\Q),
        \qquad
        G(\lambda)_\Q\cap P(\Q)=Z(\Q).
        \]
        Hence, the result follows.
        \item[-] Suppose
        \[
        \lambda=
        \begin{pmatrix}
        \alpha&0\\0&\beta
        \end{pmatrix},
        \qquad \alpha\neq\beta.
        \]
        Then \(G(\lambda)_\Q=A(\Q)\), where \(A\) is the diagonal torus. If
        \[
        \gamma=
        \begin{pmatrix}
        a&b\\c&d
        \end{pmatrix},
        \]
        then
        \[
        \gamma^{-1}\lambda\gamma
        =
        (ad-bc)^{-1}
        \begin{pmatrix}
        *&*\\
        ac(\alpha-\beta)&*
        \end{pmatrix}.
        \]
        Hence \(\gamma^{-1}\lambda\gamma\notin P(\Q)\) if and only if \(a\neq 0\) and \(c\neq 0\). In that case \(\gamma\) has a unique decomposition
        \[
        \gamma=
        \begin{pmatrix}
        a&0\\0&c
        \end{pmatrix}
        \begin{pmatrix}
        1&0\\1&1
        \end{pmatrix}
        p,
        \qquad p\in P(\Q),
        \]
        so that
        \[
        G(\lambda)_\Q\backslash
        \{\gamma\in G(\Q):\gamma^{-1}\lambda\gamma\notin P(\Q)\}
        =
        \lambda_0 P(\Q)/Z(\Q),
        \qquad
        \lambda_0=
        \begin{pmatrix}
        1&0\\1&1
        \end{pmatrix}.
        \]
    \item[-] $\lambda$ unipotent: Let $\lambda = \left(\begin{smallmatrix}
    \alpha & 1 \\0 & \alpha
\end{smallmatrix}\right)$. Then $G(\lambda)_F = N_F$. Again, observe $$\gamma^{-1} \lambda \gamma = \left(\begin{matrix}
    a & b \\ c & d
\end{matrix}\right)^{-1}\left(\begin{matrix}
    \alpha & 1 \\ 0 & \alpha
\end{matrix}\right)\left(\begin{matrix}
    a & b \\ c & d
\end{matrix}\right) = \left(\begin{matrix}
    a & b \\ c & d
\end{matrix}\right) = (ad-bc)^{-1} \left(\begin{matrix}
    * & * \\ -c^2 & *
\end{matrix}\right) $$
Thus, $\gamma^{-1} \lambda \gamma \not \in P_F$ if and only if $c \neq 0$, in which case $\gamma$ is uniquely represented as $$\left(\begin{matrix}
    1 & a/c \\  & 1
\end{matrix}\right)\left(\begin{matrix}
    0 & 1 \\ 1 & 0
\end{matrix}\right)\left(\begin{matrix}
    c & d \\ 0 & (bc-ad)/c
\end{matrix}\right)$$
Thus we see that $G(\lambda)_F \backslash \{\gamma \in G_F\: |\: \gamma^{-1} \lambda \gamma \not\in P_F\} = \lambda_0 P_F$ with $\lambda_0 = \left(\begin{smallmatrix}
    0 & 1 \\ 1 & 0
\end{smallmatrix}\right)$
    \end{itemize}
\end{proof}
This lemma simplifies \eqref{eq:KC} to 
\begin{equation*}
    \mathcal{K}_{s_B,C}(x) = \sum_{p \in P(\Q)/Z(\Q)} \phi_{s_B}(x^{-1}p^{-1}\lambda_0^{-1}\lambda \lambda_0 p x)
\end{equation*}
Substituting this and the definition for $f(x,s)$, we arrive at the nice simplified expression of orbital integrals:
\begin{lemma}
    For a non-central conjugacy class $C$, the corresponding orbital integral $I_{C}(s_B,s)$ is given by 
    \begin{equation}
        I_C(s_B,s) = \int_{G(\A)} \phi_{s_B}(x^{-1}\lambda_0^{-1} \lambda \lambda_0 x) \left(\int_{\widehat{\Z}^{\,\times}}
        \Phi((0,1)ux)d^\times u \right) |\det x|^s  dx 
    \end{equation}
\end{lemma}
\begin{proof}
    From \ref{conclassdecomp} and the definition of $f(x,s)$ \eqref{eq:f}, the integral simplifies to 
    \begin{align*}
        I_C(s_B,s) &= \int_{\XP} \sum_{p \in P(\Q)/Z(\Q)} \phi_{s_B}(x^{-1}p^{-1}\lambda_0^{-1}\lambda \lambda_0 p x) \int_{\A^{\times}} \Phi((0,t)x) |\det x|^{s} |t|^{2s} d^{\times} t \: dx \\
        &= \int_{Z_{\infty}^{+}Z(\Q) \backslash G(\A)} \phi_{s_B}(x^{-1}\lambda_0^{-1}\lambda \lambda_0  x) |\det x|^{s} \int_{\A^{\times}} \Phi((0,t)x) |t|^{2s} d^{\times} t \: dx
    \end{align*}
    Using the decomposition $\A^{\times} = \A^{\times}_1 \times \R^{\times}_{>0}$ and $\A^{\times}_1 \cong \Q^{\times} \times \widehat{\Z}^{\,\times}$ gives 
    \[
    I_C(s_B,s) = \int_{G(\A)} \phi_{s_B}(x^{-1}\lambda_0^{-1} \lambda \lambda_0 x) \left(\int_{\widehat{\Z}^{\,\times}}
        \Phi((0,1)ux)d^\times u \right) |\det x|^s  dx 
    \]
\end{proof}
For $\Phi$ as in \eqref{Phi}, the orbital integral simplifies further. We record it as corollary here for future reference.
\begin{corollary}
    For $\Phi$ given by
    \[
        \Phi_{\infty} (x,y) = e^{-(x^2+y^2)}, \qquad \Phi_q(x,y) = \mathbf{1}_{\Z_q} \times \mathbf{1}_{\Z_q}
    \]
    the orbital integral for any non-central conjugacy class \(C\) is given by 
    \begin{equation}
        I_C(s_B,s) = \int_{G(\A)} \phi_{s_B}(x^{-1}\lambda_0^{-1} \lambda \lambda_0 x) \Phi((0,1)x) |\det x|^s  dx \label{eq:orbint}
    \end{equation}
\end{corollary}
\begin{remark} \label{eulerproductorbint}
    This form of the orbital integral has a advantage that it factors easily into local factors as
    \begin{align*}
        I_C(s_B,s) &= \int_{G(\A)} \phi_{s_B}(x^{-1}\lambda_0^{-1} \lambda \lambda_0 x) \Phi((0,1)x) |\det x|^s  dx \\
        &= \prod_{\nu} \int_{G_{\nu}}  \phi_{s_B,\nu}(x^{-1}_{\nu}\lambda_{0,\nu}^{-1} \lambda_{\nu} \lambda_{0,\nu} x_\nu) \Phi((0,1)x_{\nu}) |\det x_{\nu}|^s  dx
    \end{align*}
\end{remark}
We recall that in this thesis, we begin the study of the elliptic part of this trace formula $I_{\Ell}(s_B,s)$ for the $k$th symmetric power representation of $\GL_2(\C)$, which is given by 
\begin{equation}
    T^k_{\Ell}(s_B,s) = \sum_{d \in \Q^{\times}} \sum_{\substack{t \in \Q \\ t^2 - 4d \neq \square}} I_{(t,d)}(s_B,s)
\end{equation}
Here, a pair \((t,d)\) of the trace and determinant denotes the conjugacy class \(C\) of \(\left( \begin{smallmatrix}
    t & - d \\ 1 & 0
\end{smallmatrix} \right)\).
\begin{remark}
    This elliptic part essentially studies the elliptic part of the $P$ action on $G$ by conjugation, whereas the elliptic part of the Selberg trace formula studies the elliptic class of $G$-conjugacy classes of $G$.
\end{remark}
For the rest of the paper, we work with the choices mentioned in the beginning of this section \ref{testfnchoice}.

\section{Archimedean local orbital integral}
We work over \(\R\) in this section. Consider an element $\lambda = \left( \begin{matrix}
    t & -d \\ 1 &0
\end{matrix} \right)$ such that $t^2-4d \neq \square$ determined by a pair \((t,d)\). The archimedean local orbital integral is given by 
\[
I_{\lambda, \infty}(s_B,s) = \int_{G(\R)} \phi_{\infty}(g^{-1} \lambda g) \Phi_{\infty}((0,1)g) |\det g|^{s} \: dg
\]
with the functions appearing given by \ref{testfnchoice}. Note that \(\lambda_0\) is identity because \(\lambda\) is chosen to be elliptic. We record some useful observations here
\begin{itemize}\label{archob}
    \item[-] \(\phi_{\infty}\) is only supported on the positive determinant (\(d>0\)) locus and is $K$-invariant, that is \( \phi_{\infty}(gk) = \phi_{\infty}(g)\) for $K = \text{SO}(2)$.
    \item[-] \(\Phi_{\infty}(x,y) = e^{-x^2 + y^2}\), and hence, \( \Phi_{\infty}((0,1)gank) = \Phi_{\infty}((0,1)g)\) for $a \in A$ and $n \in N$. Here, \(A = \left\{ \begin{pmatrix}
        a & \\ & 1
    \end{pmatrix} : a \in \R^{\times} \right\}\) and \(N = \left\{ \begin{pmatrix}
        1 & n \\ &1 
    \end{pmatrix}  : n \in \R\right\}  \)
\end{itemize}
These observations along with the Iwasawa decomposition $G(\R) = ZANK$ simplifies the above integral to 
\begin{align*}
    I_{\lambda, \infty}(s_B,s) &= \int_{AN} \phi_{\infty}(n^{-1}a^{-1} \lambda an) \left( \int_{Z} \Phi_{\infty}((0,1)zan) |z|^{2s} d^{\times} z\right) |a|^{s} d^{\times}a\: dn \\
    &= \pi^{-s} \Gamma(s) \int_{\R} \int_{\R^{\times}} \phi_{\infty}\left( \begin{pmatrix}
        t-xy & tx-y^{-1}d-yx^2 \\ y &yx 
    \end{pmatrix} \right) |y|^{s} d^{\times}y \: dx \\
    &= \pi^{-s} \Gamma(s)  \frac{(2 \kappa -1) (2i)^{2\kappa}}{4\pi}\int_{\R}\int_{\R^{\times}} \frac{d^{\kappa} |y|^s}{(-tx + y^{-1}d +yx^2 + y + it)^{2\kappa}} d^{\times}y \: dx \\
    &= \pi^{-s} \Gamma(s)  \frac{(2 \kappa -1) (2i)^{2\kappa}}{4\pi} \int_{\R} \int_{\R} \frac{y^{2\kappa -2} |y|^{s}}{(x^2-\frac{t^2}{4}+d+y^2 +ity)^{2\kappa}} dx \: dy
\end{align*}
where we have used the substitution \(x \mapsto y^{-1}(x+\frac{t}{2})\) to get the last expression. Now, splitting the $y$-integral into positive and negative parts; substituting \(x \mapsto \sqrt{d} x\), \(y \mapsto \sqrt{d}y\) and setting $\alpha = \frac{t}{2\sqrt{d}}$, the integral becomes
\begin{equation}
    \pi^{-s} \Gamma(s)  \frac{(2 \kappa -1) (2i)^{2\kappa} d^{\frac{s}{2}}}{4\pi} \left( I^{+}_{\lambda, \infty}(s_B,s) + I^{-}_{\lambda, \infty}(s_B,s) \right) \label{eq:archsim1}
\end{equation}
with 
\begin{equation}
I^{\pm}_{\lambda, \infty}(s_B,s) = \int_{0}^{\infty} \int_{\R} \frac{y^{2\kappa - 2+s}}{(x^2+1-\alpha^2+y^2\pm 2i\alpha y)^{2\kappa}} dx \: dy \label{eq:Ipm}
\end{equation}
We compute these integrals in the lemma below to get the archimedean orbital integral.
\begin{proposition}[Archimedean orbital integrals] \label{archorbint}
    For an element \((t,d)\) such that $t^2 - 4d \neq \square$ and $d > 0$; the archimedean orbital integral \(I_{\lambda, \infty}(s_B,s)\) is given by 
    \[
     C_{\kappa}(s) \, d^{\frac{s}{2}} P_{\kappa}\left(s,\frac{t}{2\sqrt{d}}\right)  
    \]
    with
    \[
    C_{\kappa}(s) = \frac{\Gamma(s) (2k-1) (-1)^{\kappa} \Gamma(2k - \frac{1}{2}) \Gamma(2k+s-1) \Gamma(2k-s) }{4^{1-\kappa} \pi^{s+\frac{1}{2}}\Gamma(2k) \Gamma(4k-1)}
    \]
    and 
    \[
    P_{\kappa}(s,\alpha) = (1-\alpha^2)^{\frac{1}{2}-2\kappa} \left(\sum_{\pm} (\pm i(\alpha+1))^{2\kappa+s-1} \F\left(2\kappa-\frac{1}{2}, 2\kappa+s-1; 4\kappa -1; \frac{2}{1-\alpha} \pm \text{sgn}(\alpha ) i0\right) \right)
    \]
\end{proposition}
\begin{proof}
   Note that the \(x\)-integrals in \(I_{\lambda,\infty}^{\pm}\) of \eqref{eq:Ipm} are of the form \(\int_{\R} \frac{dx}{(x^2+C)^m}\) with \(C \in \C \backslash (-\infty,0]\). For \(\mathfrak{R}(C) > 0\), using the Mellin transform identity 
   \[
   (x^2+C)^{-m} = \frac{1}{\Gamma(m)} \int_{0}^{\infty} t^{m-1} e^{-t(x^2+C)} dt
   \]
   we get that 
   \[
   \int_{\R} \frac{dx}{(x^2+C)^m} = \sqrt{\pi} \frac{\Gamma(m-\frac{1}{2})}{\Gamma(m)} C^{\frac{1}{2}-m}
   \]
   By analytic continuation, this remains valid for all $C \not \in (-\infty, 0]$. Using this, we get
   \begin{align*}
       I_{\lambda,\infty}^{\pm}(s_B,s) &= \sqrt{\pi} \frac{\Gamma(2\kappa-\frac{1}{2})}{\Gamma(2\kappa)} \int_{0}^{\infty} y^{2\kappa -2 +s} (y^2+ 1-\alpha^2 \pm 2i\alpha y)^{\frac{1}{2}-2\kappa} dy \\
       &= \sqrt{\pi} \frac{\Gamma(2\kappa-\frac{1}{2})}{\Gamma(2\kappa)} \int_{0}^{\infty} y^{p-1} (y+r_{\pm})^{-q} (y+s_{\pm})^{-q} dy
   \end{align*}
   for \(r_{\pm} = \pm i(\alpha+1)\), \(s_{\pm} = \pm i(\alpha -1)\), \(p = 2\kappa + s\) and \(q = 2\kappa-\frac{1}{2}\). By substituting the standard Euler-type integral representation for the hypergeometric function,
   \[
   \int_{0}^{\infty} y^{p-1} (y+a)^{-q} (y+b)^{-q} dy = a^{p-q} b^{-q} B(p, 2q-p) \F(q,p;2q;1-\frac{a}{b})
   \]
   into \eqref{eq:archsim1}, the claim follows.
\end{proof}
\begin{remark}
\begin{enumerate}
    \item For $\alpha^2 < 1$, the argument of the hypergeometric function \(\frac{2}{1-\alpha} \in [1,\infty)\) lies on the principal branch. At these points, the we define hypergeometric function by averaging, that is,
    \[
    \F(a,b;c;x) = \frac{1}{2} \left( \F(a,b;c;x+i0) + \F(a,b;c;x-i0)\right) \qquad x\in [1,\infty)
    \]
    \item \label{simplified_prefactor} For the case when $s \in \R$, since the second term becomes conjugate to the first term, the prefactor $P_{\kappa}(s,\alpha)$ simplifies considerably as follows:
    \[
    P_{\kappa}(s,\alpha) = (1-\alpha^2)^{\frac{1}{2}-2\kappa} \: \Re \left( ( i(\alpha+1))^{2\kappa+s-1} \F\left(2\kappa-\frac{1}{2}, 2\kappa+s-1; 4\kappa -1; \frac{2}{1-\alpha} \right) \right) 
    \] 
    \item \label{at1} At \(s=1\), using 
    \[
    \F(2\kappa-\frac{1}{2}, 2\kappa; 4\kappa-1;z) = \frac{1}{\sqrt{1-z}}\left(\frac{1}{2} + \frac{1}{2}\sqrt{1-z}\right)^{2-4\kappa}
    \]
    the above archimedean orbital integral becomes
    \[
    I_{\lambda,\infty}(s_B,1) = \frac{(-1)^{\kappa-1} d^{\frac{1}{2}}}{2\pi}\begin{cases}
         (\sqrt{1-\alpha^2}+i\alpha)^{1-2\kappa} + (\sqrt{1-\alpha^2}-i\alpha)^{1-2\kappa} \qquad& \alpha^2 = \frac{t^2}{4d}<1\\
         0 & \frac{t^2}{4d} >1
    \end{cases}
    \]
    This exactly matches the archimedean local orbital integral that one obtains in the usual trace formula \cite{altug2015endoscopytraceformula}. This is not surprising as integrating against Eisenstein series in Zagier's approach does not interfere with the archimedean place. (they only interact with the finite places to give the class numbers through Zagier $L$-functions \ref{globalorbint})
    \item \label{archorbsing} Note that because the hypergeometric functions \(\F\) have singularities at \(1\) and \(\infty\), the archimedean orbital integral has singularities at \(\alpha = \pm 1\), which correspond to the central elements in \(\GL_2(\Q)\).
\end{enumerate}
\end{remark}

\section{Non-Archimedean local orbital integrals}
In this section, we compute the non-archimedean local orbital integrals appearing in \ref{eulerproductorbint}. For completeness, we include all necessary details. The method was carried out in \cite{luo2022biasrootnumbershilbert} in the case of the function \(\mathbf{1}_{\Z_q}\), while the computations of the weight factors were previously obtained in \cite{Cherubini_2021} for a slightly different normalization. 
\\
Throughout this section, \(q\) denotes a fixed prime. For simplicity, we assume \(q \neq 2\), so that the quadratic étale \(\Q_q\)-algebras \(E\) are representated by the split algebra \(\Q_q \times \Q_q\), the unramified field extension \(\Q_q(\sqrt{\epsilon})\), and the two ramified field extensions \(\Q_q(\sqrt{q})\) and \(\Q_q(\sqrt{q\epsilon})\), where we have fixed the uniformizer \(q\) and \(\epsilon \in \Z_q^{\times}\) a non-square unit. We are interested in the integral 
\begin{equation*}
    I_{\lambda,q}(s_B,s) = \int_{\GL_2(\Q_q)} \phi_{s_B,q}^{k}(g^{-1}\lambda g) \Phi_q((0,1)g) |\det g|^{s} \; dg  \label{eq:qorb}
\end{equation*}
where \(\Phi_q(x,y) = \mathbf{1}_{\Z_q}(x)  \mathbf{1}_{\Z_q}(x)\) and \(\phi_{s_B,q}^{k}\) is the basic function attached to the symmetric power representation \(\sigma_k\) (\ref{kbasicfn}). We use the theory of suitable embeddings of \(E\) into \(\M_2(\Q_q)\) so that \(\GL_2(\Q_q)\) factors nicely with respect to this embedding. For more general results on embeddings, refer to \cite{luo2022biasrootnumbershilbert}. The suitable embedding are given by well-positioned embeddings defined as
\begin{definition}[Well-positioned embedding] \label{wellposembed}
    Let \(E\) be a quadratic étale extension of \(\Q_q\) with an embedding \(\iota : E \rightarrow \M_2(\Q_q)\) as \(\Q_q\)-algebras. Such an embedding is said to be well-positioned with respect to an element \(\theta \in \mathcal{O}_E\) and \(\theta \not \in \Q_q\) if \[
    \theta^2 - b\theta + a = 0, \qquad \iota(\theta) = \begin{pmatrix}
        b & -a \\ 1 & 0
    \end{pmatrix}, \qquad a \not \in q^2 \Z_q
    \] 
    and \[
    \mathcal{O}_E = \Z_q[\theta]
    \]
    Here, \(\mathcal{O}_E\) is the ring of integers of \(E\).
\end{definition}
The following decomposition of \(\GL_2(\Q_q)\) is a key ingredient that enables us to understand the non-archimedean local orbital integrals ahead.
\begin{lemma}\label{gl2decomp}
    Let \(E\) be a well-positioned étale quadratic extension of \(\Q_q\) with respect to \(\theta \in \M_q(\Q_q)\) (identify \(E\) with its image \(\iota(E)\)). If \(E\) is non-split, then 
    \[
    \GL_2(\Q_q) = \bigsqcup_{r=0}^{\infty} E^{\times} a(q^{-r}) \GL_2(\Z_q), \qquad a(t):= \begin{pmatrix}
        t & 0 \\ 0 & 1
    \end{pmatrix}
    \]
    If \(E\) is split, we identify \(\theta\) with \((\lambda_1,\lambda_2) \in \Q_q \times \Q_q\), and then we have 
    \[
    \GL_2(\Q_q) = \bigsqcup_{r=0}^{\infty} \begin{pmatrix}
        \lambda_1 & \lambda_2 \\ 1 & 1
    \end{pmatrix} A\: \: n(q^{-r})\GL_2(\Z_q), \qquad n(x) =\begin{pmatrix}
        1 & x \\ 1&0
    \end{pmatrix}, \qquad A = \left\{a(t) : t \in \Q_q^{\times} \right\}
    \]
\end{lemma}
\begin{proof}
    See Proposition 4.1 \cite{luo2022biasrootnumbershilbert}
\end{proof}
We now treat the integral \eqref{eq:qorb} for the different cases of \(\lambda\) being split, non-split ramified or non-split unramified in \(\GL_2(\Q)\).
\begin{remark} \label{dsupp}
    For any of these cases and for a pair \( (t,d) \in \Q \times
    \Q \), by considering the support of the basic function \(\phi_{s_B,q}^{k}\), we observe that the integral \eqref{eq:qorb} vanishes unless \(k \mid v_q(d)\).  
\end{remark}
For the rest of this section, we assume this to be the case.

\subsection{The split case}
Let \(\lambda = \begin{pmatrix}
    t & -d \\ 1 & 0
\end{pmatrix}\) with \(t^2 -4d = f^2D\) (\(D\) fundamental discriminant) such that \( \lambda_1, \lambda_2 \in \Q_q\) are the roots of \[
X^2 - tX+d = 0
\]
We write \[
 v_q(t) =m,\qquad  v_q(f) = v_q(\lambda_1-\lambda_2) = l, \qquad v_q(d) = v_q(\lambda_1\lambda_2) = kn
\]
so that \(l \geq m\). Substituting \(g \mapsto Xg\) for \(X = \left( \begin{smallmatrix}
    \lambda_1 & \lambda_2 \\ 1 & 1
\end{smallmatrix} \right)\), using the decomposition of \(\GL_2(\Q_q)\)
\[
\GL_2(\Q_q) = \bigsqcup_{r=0}^{\infty} \begin{pmatrix}
    \lambda_1 & \lambda_2 \\ 1 & 1
\end{pmatrix} A n(q^{-r}) \GL_2(\Z_q)
\]
from \ref{gl2decomp}, the fact that the integrand is \(\GL_2(\Z_q)\)-invariant and the decomposition 
\[
N(\Q_q) = N(\Z_q) \sum_{r=0}^{\infty} N(q^{-r} \Z_q^{\times})
\]we get 
\begin{align} 
    I_{\lambda,q}(s_B,s) &= \int_{\GL_2(\Q_q)} \phi_{s_B,q}^{k}(g^{-1}\lambda g) \Phi_q((0,1)g) |\det g|^{s} \; dg  \notag\\
    &= q^{-ls}\sum_{r=0}^{\infty} C_r \wt(r) \label{splitorb}
\end{align}
with the coefficients \(C_r\) given by 
\begin{equation}
    C_r = \phi_{s_B,q}^{k} \left( n(-q^{-r}) \begin{pmatrix}
        \lambda_1 & \\ & \lambda_2
    \end{pmatrix} n(q^{-r})\right) = q^{-n(s_B+\frac{k}{2})} \sum_{j=0}^{\min(m,l)} q^j \mathfrak{c}(k,n,j) \mathbf{1}_{j \leq l-r} \label{eq:crsplit} 
\end{equation} 
and the weight factor \(\wt(r)\) being
\begin{align*}
    \wt(r) &= \int_{A(\Q_q) \times Z(\Q_q)} \Phi_q((1,1)z(u)a(y)n(q^{-r})) \: |u|^{2s} |y|^{s} d^{\times}u \; d^{\times}y \begin{cases}
    \vol(N(\Z_q), dn) \qquad & r = 0\\
    \vol(N(q^{-r}\Z_{q}^{\times},dn)) & r \geq 1
\end{cases}
\end{align*}
Now, with the normalization of measure we are following, \(\vol(N(\Z_q),dn) = 1\) and \(\vol(N(q^{-r}\Z_q^{\times}), dn) =q^r(1-q^{-1})\). Expanding the integral and using the substitutions \(u \mapsto u/y\) and \(y \mapsto y^{-1}\), we get 
\begin{align*}
    &= \int_{\Q^{\times}_q \times \Q_q^{\times}} \mathbf{1}_{\Z_q}(u) \mathbf{1}_{\Z_q}(u(q^{-r}+y)) |u|^{2s} |y|^s d^{\times}u d^{\times} y\\
    &= \sum_{v_q(u) = j \geq 0} q^{-2sj} \begin{cases}
        q^{sj} (1-q^{-s}) \qquad & j \geq r\\
        q^{sr} \int_{y\in 1+q^{r-j}\Z_q} |y|^s d^{\times}y & r > n \geq 0
    \end{cases} \\
    &= (1-q^{-s})^{-2} \begin{cases}
        1 \qquad &r =0\\
        \frac{(q^{-1/2}Z)^r(qZ + Z^{-1} -2q^{1/2}) - (q^{-1/2}Z^{-1})^{-r}(Z + qZ^{-1} -2q^{1/2})}{(1-q^{-1})(Z-Z^{-1})} & r> 0
    \end{cases}
\end{align*}
where \(Z = q^{s-\frac{1}{2}}\). Thus, the weights for the split case are:
\begin{equation}
    \wt(r) = (1-q^{-s})^{-2} \begin{cases}
        1 & r = 0\\
        \frac{(q^{1/2}Z)^r(Z + q^{-1}Z^{-1} -2q^{-1/2}) - (q^{1/2}Z^{-1})^{-r}(Z^{-1} +q^{-1}Z-2q^{-1/2})}{(Z-Z^{-1})} & r >0 
    \end{cases} \label{wtsplit}
\end{equation}

\subsection{The non-split case}
Let \(\lambda = \begin{pmatrix}
    t & -d \\ 1 & 0
\end{pmatrix}\) with \(t^2 -4d = f^2D\) be non-split, so that the corresponding quadratic extension is \(\Q_q(\sqrt{q})\), \(\Q_q(\sqrt{q\epsilon})\) or \(\Q(\sqrt{\epsilon})\). Let \(e(E \mid \Q_q)\) denote the ramification index (so \(e(E\mid \Q_q) = 1\) for the unramified case and \(e(E \mid \Q_q) = 2\) for ramified case). Let \(\theta \in E\) be such that \(E\) is a well-positioned embedding with respect to \(\theta\). It is straightforward to see that 
\[ \mathcal{O}_{\Q_q(\sqrt{q})} = \Z_q[\sqrt{q}], \qquad \mathcal{O}_{\Q_q{\sqrt{q\epsilon}}} = \Z_q[\sqrt{q\epsilon}], \qquad \mathcal{O}_{\Q_q(\sqrt{\epsilon})} = \Z_q[\sqrt{\epsilon}]\] and hence, from \ref{wellposembed}, we have the decomposition 
\[
\GL_2(\Q_q) = \bigsqcup_{r=0}^{\infty} E^{\times} a(q^{-r}) \GL_2(\Z_q)
\]
Let \( t = 2a\) and \(t^2 - 4d = b^2 q\) and 
\[ v_q(t) = m, \qquad v_q(f) = l, \qquad v_q(d) = kn 
\] Then substituting \(g \mapsto Xg\) for \(X = \begin{pmatrix}
    b & a \\ 0 & 1
\end{pmatrix}\) and using the above decomposition along with the fact that the integrand is \(\GL_2(\Z_q)\)-invariant, we get
\begin{equation}
    I_{\lambda,q}(s_B,s) = q^{-ls} \sum_{r = 0}^{\infty} C_r \wt(r) \label{eq:nonplitramifiedord}
\end{equation}
with 
\begin{equation}
    C_r = \phi_{s_B,q}^{k} \left(a(q^r) \begin{pmatrix}
        a & bq \\ b &a
    \end{pmatrix} a(q^{-r}) \right) \label{eq:crramified}
\end{equation}
and 
\begin{equation}
    \wt(r) = q^{-rs} d_r \int_{E^{\times}} \mathbf{1}_{\mathcal{O}_r}(e) |\text{Nm}_{E/\Q_q}(e)|^s d^{\times}e \label{eq:wtramified}
\end{equation}
having chosen a standard embedding of \(E\) into \(M_2(\Q_q)\) so that every element can be written as \(e = y + x\theta\); and with \(\mathcal{O}_r = \Z_q + q^r\Z_q \theta\) and \(d_r = \vol\left(\mathcal{O}_{E}^{\times}/(\mathcal{O}_{r}\cap \mathcal{O}_E^{\times}) \right)\) coming from the correct normalization of measures. These weight factors are given by the following lemma:
\begin{lemma} \label{nonsplitweights}
    The weight factors in \eqref{eq:wtramified} are given by:\\
    In the ramified case:
    \[
    \wt(r) = \begin{cases}
        (1-q^{-s})^{-1}\qquad \qquad r = 0\\
        (1-q^{-s})^{-1} \frac{(q^{1/2}Z)^{r} (Z - q^{-1/2}) - (q^{1/2}Z^{-1})^{r}(Z^{-1} - q^{-1/2})}{Z- Z^{-1}} & r>0
    \end{cases}
    \]
    In the unramified case:
    \[
    \wt(r) = \begin{cases}
        (1-q^{-2s})^{-1}\qquad \qquad r = 0\\
        (1-q^{-2s})^{-1} \frac{(q^{1/2}Z)^{r} (Z - q^{-1}Z^{-1}) - (q^{1/2}Z^{-1})^{r}(Z^{-1} - q^{-1}Z)}{Z- Z^{-1}} & r>0
    \end{cases}
    \]
    with \(Z = q^{s-\frac{1}{2}}\)
\end{lemma}
\begin{proof}
    The proof involves careful calculation of \(d_r\) and splitting the integral over \(E^{\times}\) as a sum over valuation shells. For details, see Section 4.2 of \cite{Cherubini_2021}.
\end{proof}

\subsection{A uniform expression for the non-archimedean local orbital integrals} \label{Uniformexp}
Observe that for a fixed determinant valuation \( v_q(d) = kn\), a matrix \(\begin{pmatrix}
    \alpha & \beta \\ \gamma & \delta
\end{pmatrix} \in \GL_2(\Q_q)\) with \(v_q(\alpha\delta-\beta\gamma) = kn\) belongs to the support of \(T_{k,q}^{n,j}\) if and only if 
\[
0 \leq  
j \leq \min \left(v_q(\alpha), v_q(\beta), v_q(\gamma), v_q(\delta)\right)
\]
and in the case it does, it contributes \(q^j \mathfrak{c}(k,n,j)\) towards \(\phi_{s_B,q}^{k}\left(\begin{smallmatrix}
    \alpha & \beta \\ \gamma & \delta
\end{smallmatrix}\right)\).
This observation gives 
\begin{equation}
    C_r = q^{-n(s_B+\frac{k}{2})} \sum_{j=0}^{\min(m,l)} q^j \: \mathfrak{c}(k,n,j) \mathbf{1}_{j \leq l-r}
\end{equation}
for all the cases: split as well as non-split.\\
Moreover, the weight factors from the earlier section can be expressed uniformly through the simple form
\begin{equation}
    \wt(r) = \begin{cases}
    (1-q^{-s})^{-1} (1 - \chi_D(q)q^{-s})^{-1} \qquad \qquad &r=0\\
    (1-q^{-s})^{-1} (1 - \chi_D(q)q^{-s})^{-1} \left(B_{r} -B_{r-1}\right)& r > 0
\end{cases} \label{eq:wts}
\end{equation}
with \[
B_{r} = \frac{q^{\frac{r}{2}}(Z^{r+1} - Z^{-(r+1)}) - \chi_{D}(q) q^{\frac{r-1}{2}} (Z^{r} -Z^{-r})}{Z- Z^{-1}}
\]
where \(\chi_D\) denotes the Legendre symbol \( \left( \frac{D}{\cdot} \right) \). We finally use these results and summarize the non-archimedean local orbital integral \ref{eq:qorb} below
\begin{proposition}[Non-archimedean orbital integrals]\label{qorbcomputed} 
    At a finite place \(q\), the local orbital integral 
    \[
    I_{\lambda,q}(s_B,s) = \int_{\GL_2(\Q_q)} \phi_{s_B,q}^{k}(g^{-1}\lambda g)\Phi_q((0,1) g) |\det g|^{s} \: dg
    \]
    vanishes unless \(v_q(t) \geq 0\) and \(v_q(d) = kn \geq 0\); and when this is the case, is given by
    \[
    q^{-n(s_B+\frac{k}{2})}\zeta_{D,q}(s) \sum_{j=0}^{\min(m,l)} q^{-j(s-1)}\mathfrak{c}(k,n,j) L_{q^{l-j}}^{(D)}(s)
    \]
    with 
    \begin{align*}
        \zeta_{D,q}(s) &= (1-q^{-s})^{-1} \: (1-\chi_D(q)q^{-s})^{-1}\\
        L_{q^{r}}^{(D)}(s) &= Z^{-r} \frac{(Z^{r+1}-Z^{-(r+1)}) -\chi_D(q)(Z^r - Z^{-r})}{Z - Z^{-1}}, \qquad Z = q^{s- \frac{1}{2}}
    \end{align*}
\end{proposition}
\begin{proof}
    We can now treat all the split and non-split cases together. The telescoping nature of the weight factors in \eqref{eq:wts} simplifies the computations immensely 
    \begin{align*}
        &I_{\lambda,q}(s_B,s) = q^{-ls} \sum_{r=0}^{\infty} C_r \wt(r) \\
        &=q^{-n(s_B+\frac{k}{2})} \zeta_{D,q}(s) q^{-ls} \left[ \sum_{j=0}^{\min(m,l)} q^{j} \mathfrak{c}(k,n,j)  + \sum_{j=0}^{\min(m,l)} q^j \mathfrak{c}(k,n,j) \sum_{r=1}^{l-j} (B(r) - B(r-1)) \right] \\
        &= q^{-n(s_B+\frac{k}{2})} \zeta_{D,q}(s) q^{-ls} \sum_{j=0}^{\min(m,l)} q^{j} \mathfrak{c}(k,n,j)  B(l-j)\\
        &= q^{-n(s_B+\frac{k}{2})} \zeta_{D,q}(s) \sum_{j=0}^{\min(m,l)} q^{-j(s-1)} \mathfrak{c}(k,n,j) L_{q^{l-j}}^{(D)}(s)
    \end{align*}
    as claimed.
\end{proof}
\begin{remark}
    The factors \((1-\chi_D(q)q^{-s})^{-1}\:L_{q^{r}}^{(D)}(s)\) are exactly the Euler factors for the Zagier \(L\)-function, which when evaluated at \(s=1\) give the class numbers that represent the volume terms. Thus we can interpret the Zagier approach as a continuous analytic deformation of the beyond endoscopy trace formula.
\end{remark}
As the Euler factors for the Zagier \(L\)-functions appear in the local orbital integrals, we now study the Zagier \(L\)-function.
\section{Zagier \(L\)-function and their theta series}
In this section, we develop the theta series attached to the Zagier \(L\)-function and prove its functional equation. We first recall the classical theta series attached to a primitive quadratic Dirichlet character of a fundamental discriminant. We then extend the construction to non-primitive quadratic characters of discriminant \(Dr^2\). Finally, for a general discriminant \(\Delta = Dl^2\), we construct a theta series whose Mellin transform is the Zagier \(L\)-function \(L(s,\Delta)\), and conclude with its functional equation.
\subsection{Theta series for a fundamental discriminant}
Let \(D\) be a fundamental discriminant, and let \(\chi_D\) denote the associated primitive quadratic Dirichlet character. We write
\[\delta_D \in \{0,1\} \text{ such that } \chi_D(-1) = (-1)^{\delta_D} \]
The classical Dirichlet \(L\)-function associated with \(D\) is given by \[
L(s, \chi_D) = \sum_{n=1}^{\infty} \frac{\chi_D(n)}{n^s}, \qquad \mathfrak{R}(s) > 1
\]
Using the identity 
\begin{equation} 
    \int_{0}^{\infty} x^{(s+\delta_D)/2} e^{-\pi n^2 x} d^{\times}x = (\pi n^2)^{-(s + \delta_D)/2}\Gamma\left( \frac{s+\delta_D}{2}\right) \label{eq:gamma}
\end{equation}
we can write 
\[
L(s,\chi_D) = \frac{\pi^{(s+\delta_D)/2}}{\Gamma\left(\frac{s+\delta_D}{2} \right)} \int_{0}^{\infty} \theta_D(x) x^{(s+\delta_D)/2} d^{\times}x, \qquad \qquad \mathfrak{R}(s) > 1
\]
with \(\theta_D\) given by
\[
\theta_D(x) = \sum_{n=1}^{\infty} \chi_D(n) n^{\delta_D} e^{-\pi n^2 x}
\]
The following lemma gives the functional equation for \(\theta(x)\):
\begin{lemma} \label{fneqD}
    The theta series \(\theta_D(x)\) defined as above, 
    \[
    \theta_D(x) = \frac{1}{(|D|x)^{\delta_D+\frac{1}{2}}} \theta_D\left( \frac{1}{D^2x}\right)
    \]
\end{lemma}
\begin{proof}
    In order to be able to perform Poisson summation, we extend the sum over all integers by setting 
\[ 
\Theta_D(x) := \sum_{n \in \Z} \chi_D(n) n^{\delta_D} e^{-\pi n^2 x} = 2\theta_D(x)
\]
It is enough to show the identity for \(\Theta_D(x)\). By the periodicity of the character \(\chi_D\) modulo \(|D|\), and then applying Poisson summation to the inner summation, we get 
\begin{align*}
    \Theta_D(x) &= \sum_{a \mod |D|} \chi_D(a) \sum_{m \in \Z} (m|D|+a)^{\delta_D} e^{-\pi(m|D|+a)^2x} \\
    &= \sum_{a \mod |D|} \chi_D(a) \sum_{t \in \Z} \frac{(-i)^{\delta_D}}{|D|^{\delta_D+1} x^{\delta_D+1/2}} e^{2 \pi i at/|D|} t^{\delta_D} e^{-\pi t^2/D^2x} \\
    &= \frac{(-i)^{\delta_D}}{|D|^{\delta_D+1} x^{\delta_D+1/2}} \sum_{t \in \Z} \left( \sum_{a \mod |D|}\chi_D(a)e^{2\pi i at/|D|}\right) t^{\delta_D} e^{-\pi t^2/D^2x}
\end{align*}
Now the term in the bracket simplifies to \(\Bar{\chi}_D(t)\tau(\chi_D)\) with \(\Bar{\chi}_D\) being the conjugate Dirichlet character and \(\tau(\chi_D)\) the Gauss sum. The claim follows from the fact that 
\[
\tau(\chi_D) = i^{\delta_D}\sqrt{|D|}
\]
for a fundamental discriminant \(D\).
\end{proof}
This lemma will be key ingredient in developing the theta series properties for the Zagier \(L\)-function. From the lemma, it follows that the theta series \(\theta_D(x)\) is rapidly decreasing at both \(x \rightarrow \infty\) and \(x \rightarrow 0\). Now, we treat the Zagier L-functions.
\subsection{Theta series for Zagier \(L\)-function}
Consider a non-primitive character defined by \(\Delta = Dr^2\), with \(D\) fundamental. Then the Dirichlet series corresponding to \(\chi_{\Delta}\) satisfies 
\begin{equation}
    L(s, \chi_{\Delta}) = L(s, \chi_D) \sum_{u \mid r} \mu(u) \chi_D(u) u^{-s} \label{eq:Dr^2def}
\end{equation}
We now recall the definition of the Zagier \(L\)-function.
\begin{definition}[Zagier \(L\)-function] \label{ZagLfndef}
For a discriminant \(\Delta = Dl^2\) with \(D\) being fundamental, the Zagier \(L\)-function is defined via Euler product as 
\[
L(s,\Delta) = L(s,\chi_D) \prod_{q^k \mid \mid l} q^{(\frac{1}{2}-s)k} \frac{\left(q^{(s-\frac{1}{2})(k+1)} - q^{(\frac{1}{2}-s)(k+1)} \right) - \chi_D(q) q^{-\frac{1}{2}} \left(q^{(s-\frac{1}{2})k} - q^{(\frac{1}{2}-s)k}\right)}{q^{s-\frac{1}{2}} - q^{\frac{1}{2}-s}}
\]
\end{definition}
This \(L\)-function first appeared in Zagier's work \cite{Zagier1977ModularFW} in the context of constructing modular forms whose Fourier coefficients are given by infinite sums of zeta functions of quadratic fields at arbitrary complex arguments. They also appear in the Bykovskii's work \cite{bykovskii1994density} as a Dirichlet series as follows: \\
For a discriminant \(\Delta\) and a natural number \(n\), define
\[
\rho_{n}(\Delta) = \#\{x \mod 2n : x^2 \equiv \Delta \mod 4n \}
\]
then the Zagier \(L\)-function is given by
\[
L(s,\Delta) = \frac{\zeta(2s)}{\zeta(s)} \sum_{n=1}^{\infty} \frac{\rho_{n}(\Delta)}{n^s} 
\]
Soundararajan and Young \cite{Soundararajan_2013} used this \(L\)-function to improve the prime geodesic theorem bounds. \\
Getting back to the study of Zagier \(L\)-function, from the above definition, for a discriminant \(\Delta = Dl^2\), it follows that
\begin{equation}
    L(s, \Delta) = \sum_{df^2 = \Delta} L(s,\chi_{\Delta}) f^{1-2s} \label{eq:zaglsum}
\end{equation}
For detailed proof, see Lemma 2.1 of \cite{Soundararajan_2013}. 
\begin{lemma}[Theta series for Zagier \(L\)-function] \label{zagtheta}
For \(L(s,\Delta)\) with \(\Delta = Dl^2\) and \(\mathfrak{R}(s) > 1\), we have 
\[
L(s,\Delta) = \frac{\pi^{(s+\delta_D)/2}}{\Gamma(\frac{s+\delta_D}{2})} \int_{0}^{\infty} \vartheta_{\Delta}(x) x^{(s+\delta_D)/2} d^{\times}x
\] 
for 
\begin{equation}
    \vartheta_{\Delta}(x) = \sum_{r \mid l} \left(\frac{l}{r} \right)^{2\delta_D+1} \sum_{u \mid r} \mu(u) u^{\delta_D} \chi_D(u) \theta_{D}\left( \frac{u^2l^4x}{r^4} \right) \label{eq:theta}
\end{equation}
\end{lemma}
\begin{proof}
    The description of the Zagier \(L\)-function \eqref{eq:zaglsum} and \eqref{eq:Dr^2def} gives
    \begin{align*}
        &L(s,\Delta) = \sum_{r \mid l} L(s, \chi_{Dr^2}) \left(\frac{l}{r} \right)^{1-2s} \\
        &= \frac{\pi^{(s+\delta_D)/2}}{\Gamma(\frac{s+\delta_D}{2})} \sum_{r \mid l} \left(\frac{l}{r} \right)^{1-2s} \sum_{u \mid r} \mu(u)\chi_D(u) u^{-s} \int_{0}^{\infty} \theta_D(x) x^{(s+\delta_D)/2} d^{\times}x
    \end{align*}
    Using the change of variables \(x \mapsto \frac{l^4u^2}{r^4} x\), the lemma follows.
\end{proof}
Another form of this theta series \(\vartheta_{\Delta}(x)\) that will be useful for us is the following:
\begin{lemma} \label{thetaseriesanotherform}
For \(\vartheta_{\Delta}(x)\) as defined by \ref{eq:theta}, 
\[
\vartheta_{\Delta}(x) = \sum_{f^2 \mid \Delta}^{'} f^{2\delta_D+1} \sum_{m=1}^{\infty} m^{\delta_D} \chi_{\Delta/f^2}(m) e^{-\pi m^2f^4 x}
\]
where \('\) denotes sum over \(f\) such that \(\Delta/f^2\) is a discriminant.
\end{lemma}
\begin{proof}
    Applying the change of variables \(f = \frac{l}{r}\) and expanding \(\theta_D\), we arrive at
    \begin{align*}
        \vartheta_\Delta(x) = \sum_{f^2 \mid \Delta}^{'} f^{2\delta_D+1} \sum_{m=1}^{\infty} \chi_D(m) m^{\delta_D}\left(\sum_{u \mid \left( \frac{l}{f}, m\right)} \mu(m) \right) e^{-\pi m^2f^4 x}
    \end{align*}
    Now the identity follows using Mobius inversion and the fact that \[\chi_{Dr^2}(m) = \begin{cases}
        \chi_D(m) \qquad &(m,r) = 1 \\
        0 & \text{else}
    \end{cases}\]
\end{proof}
We end this section with the following theorem that sets up the functional equation satisfied by \(\vartheta_{\Delta}(x)\).
\begin{theorem}[Functional equation for theta series of Zagier \(L\)-function]\label{thetaseriesfneq}
\(\vartheta_{\Delta}(x)\) defined as in \eqref{eq:theta} for \(\Delta = Dl^2\) satisfies the functional equation given by
\begin{equation}
    \vartheta_{Dl^2}(x) = \frac{1}{(|D|l^2x)^{\delta_D+1/2}} \sum_{n=1}^{\infty} v_{n,l} \: \chi_D\left(\frac{n}{v_{n,l}^2}\right) \: n^{\delta_D} e^{-\frac{\pi n^2}{D^2l^4x}} \label{eq:fneq}
\end{equation}
where 
\begin{equation*}
    v_{n,l} = \begin{cases}
        0 \qquad & \text{there does not exist any \(v\) such that }\left(\frac{l}{v},\frac{n}{v^2} \right) =1 \\
        v_{n,l} & \text{ unique \(v_{n,l}\) such that } \left(\frac{l}{v},\frac{n}{v^2} \right) =1 
    \end{cases}
\end{equation*}
\end{theorem}
\begin{proof}
    Using the definition of \(\vartheta_{Dl^2}(x)\) (\ref{eq:theta}) and the functional equation for \(\theta_D\) (\ref{fneqD}), we find that 
    \begin{align*}
        \vartheta_{Dl^2}(x) = \frac{1}{(|D|l^2x)^{\delta_D+1/2}} \sum_{r \mid l} r^{2\delta_D+1} \sum_{u \mid r} \mu(u) \chi_D(u) u^{-(\delta_D+1)}\theta_D\left(\frac{r^4}{D^2u^2l^4x}\right)
    \end{align*}
    Setting \(r = uv\) and then expanding \(\theta_{D}\), one arrives at
    \[
    \frac{1}{(|D|l^2x)^{\delta_D+1/2}} \left(\sum_{\substack{r \mid l \\ r \mid n}} \sum_{\substack{v \mid r \\ v \mid n/r}}v \:\mu\left(\frac{r}{v}\right)\chi_D\left(\frac{n}{v^2}\right) \right) e^{-\frac{\pi n^2}{D^2l^4x}}
    \]
    Using the substitution \(r=vm\), the sum in the bracket simplifies using Mobius inversion as follows
    \begin{align*}
        \sum_{\substack{r \mid l \\ r \mid n}} \sum_{\substack{v \mid r \\ v \mid n/r}}v \:\mu\left(\frac{r}{v}\right)\chi_D\left(\frac{n}{v^2}\right) = \sum_{\substack{v \mid l \\ v^2 \mid n}} \sum_{\substack{m \mid l/v \\ m \mid n/v^2}} v \chi_D\left(\frac{n}{v^2}\right) \mu(m) = v_{n,l} \chi_D\left(\frac{n}{v_{n,l}^2}\right)
    \end{align*}
    where we have again used Mobius inversion identity for the last equality.
\end{proof}

\section{Elliptic part of trace formula}
Recall that the elliptic part of the trace formula is given by 
\begin{equation}
    T^k_{\Ell}(s_B,s) = \sum_{d\in\Q^{\times}} \sum_{\substack{t\in \Q \\ t^2 - 4d \neq \square}} I_{(t,d)}(s_B,s) \label{eq:Tell}
\end{equation}
From \ref{archorbint} and \ref{qorbcomputed}, we know that for \(\lambda = (t, d) \in \Q \times \Q^{\times}\) with \(t^2 -4d = f^2 D\), with \(D\) fundamental discriminant
\begin{itemize}
    \item[-] \(I_{(t,d)}(s_B,s) = I_{\lambda, \infty}(s_B,s) \times \prod_{q} I_{\lambda, q}(s_B,s) \)
    \item[-] \(I_{\lambda,\infty}(s_B,s) = 0\) unless \(d > 0\). In case when \(d > 0\); from \ref{archorbint} we have 
    \[
    I_{\lambda,\infty} = d^{\frac{s}{2}} I_{\infty}^{s_B,s}\left(\frac{t}{2\sqrt{d}}\right)
    \]
    with \[
    I_{\infty}^{s_B,s}(\alpha) =  C_{\kappa}(s) P_{\kappa}\left(s,\alpha \right) 
    \] where the components are as in \ref{archorbint}. 
    \item[-] \(I_{\lambda,q}(s_B,s) = 0\) unless \(t \in \Z_q \: \forall q\) and \(d \in \Z_q^k \: \forall q\); and when these conditions are satisfied \[
    I_{\lambda,q}(s_B,s) =q^{-n(s_B+\frac{k}{2})} \zeta_{D,q}(s) \sum_{j=0}^{\min(m,l)} q^{-j(s-1)} \mathfrak{c}(k,n,j) L_{q^{l-j}}^{(D)}(s) 
    \] 
    with the terms as in \ref{qorbcomputed}. Note that here, \( m = v_q(t), l = v_q(f)\) and \(n = v_q(d)\) changes as \(q\) changes.
\end{itemize}
Thus, the support of the summations in \ref{eq:Tell} is in fact restricted to \(d \in \Z_{>0}\) and \(t \in Z\). In the case when this happens, let \(d = d_0^k\). Then, from the definition of the Zagier \(L\)-function \ref{ZagLfndef} we find that 
\begin{align}
    &I_{(t,d)}(s_B,s) = I_{\lambda, \infty}(s_B,s) \times \prod_{q} I_{\lambda, q}(s_B,s) \notag \\ 
    =& d_0^{\frac{ks}{2}} I_{\infty}^{s_B,s}\left( \frac{t}{2\sqrt{d_0^{k}}}\right) \prod_{q} q^{-v_q(d_0)(s_B+\frac{k}{2})} \zeta_{D,q}(s)  \prod_{q}
    \left(\sum_{j=0}^{\min(v_q(t),v_q(f))} q^{-j(s-1)} \mathfrak{c}(k,v_q(d_0),j) L_{q^{v_q(f)-j}}^{(D)}(s) \right) \notag \\
    =& d_0^{-s_B + \frac{k}{2}(s-1)} I_{\infty}^{s_B,s}\left( \frac{t}{2\sqrt{d_0^{k}}} \right)\: \zeta(s)\: \sum_{u \mid f, \: u \mid t} u^{-(s-1)} C_k(d_0, u) L\left(s, \frac{t^2 - 4d_0^k}{u^2}\right) \label{globalorbint}
\end{align}
where 
\begin{equation}
    C_k(d,u) = \prod_q \mathfrak{c}(k, v_q(d), v_q(u))
\end{equation}
Thus, the elliptic part of the trace formula becomes 
\begin{align}
    T_{\Ell}^{k}(s_B,s) &= \zeta(s) \sum_{d = 1}^{\infty} \frac{1}{d^{s_B - \frac{k}{2}(s-1)}} \sum_{\substack{ t \in \Z \\ t^2 - 4d^k \neq \square}}  I_{\infty}^{s_B,s}\left( \frac{t}{2\sqrt{d^k}} \right) \sum_{u \mid f, \: u \mid t} u^{-(s-1)} C_k(d, u) L\left(s, \frac{t^2 - 4d^k}{u^2}\right) \label{Tellint} 
\end{align}
\subsection{A simplification}
Recall that we are interested in the meromorphic continuation of \[\Res_{s=1} T^{k}_{\Ell} (s_B,s)\] in the variable \(s_B\). As such, we break the elliptic part \eqref{Tellint} according to \(\R\)-conjugacy classes, that is 
\[
T_{\Ell}^k(s_B,s) = T_{\Ell,,\geq0}^{k}(s_B,s) + T_{\Ell, <0}^{k}(s_B,s)
\]
where \(T_{\Ell,\geq0}^{k}(s_B,s)\) denotes the sum over pairs \((t,d)\) with \(t^2 - 4d^k > 0\) and similarly, \(T_{\Ell,<0}^{k}(s_B,s)\) denotes the sum over pairs \((t,d)\) with \(t^2 - 4d^k < 0\). \\
Because \(I_{\infty}^{s_B,1}\left(\frac{t}{2\sqrt{d^k}}\right) = 0\) for \(t^2 - 4d^k > 0\)(\ref{at1}), we have that 
\[
\Res_{s=1} T^{k}_{\Ell, \geq 0} (s_B,s) = 0
\]
Hence, we will focus on 
\begin{align}
     T_{\Ell,<0}^{k}(s_B,s) &= \zeta(s) \sum_{d = 1}^{\infty} \frac{1}{d^{s_B - \frac{k}{2}(s-1)}} \sum_{\substack{ t \in \Z \\ t^2 - 4d^k \neq \square}}  F_{s}\left( \frac{t}{2\sqrt{d^k}} \right) \sum_{u \mid f, \: u \mid t} u^{-(s-1)} C_k(d, u) L\left(s, \frac{t^2 - 4d^k}{u^2}\right) \notag \\
     &= \zeta(s) \sum_{d = 1}^{\infty} \frac{1}{d^{s_B - \frac{k}{2}(s-1)}} \sum_{\substack{u \geq 1 \\ u^2 \mid d^k}} u^{-(s-1)} C_k(d, u)  H_s\left(\frac{d^k}{u^2}\right)  \label{Tellbroken} 
\end{align}
with \(F_s\) supported only on \(\R\)-elliptic elements, that is 
\[
F_s(x) = \begin{cases}
    I_{\infty}^{s_B,s}(x) \qquad & x^2 <1 \\
    0 & \text{otherwise}
\end{cases}
\]
and 
\begin{equation}
    H_s(n) = \sum_{\substack{t \in \Z \\ t^2 - 4n \neq \square}} F_{s}\left( \frac{t}{2\sqrt{n}} \right) L\left(s, t^2 - 4n\right) \label{Hsn}
\end{equation}
\begin{remark}
\begin{itemize}
    \item[-] The simplification introduced in this section is primarily for technical convenience, as it allows certain interchanges to be justified more easily. The argument would still remain valid without this simplification, though the resulting analysis would be technically more involved.
    \item[-] \(H_s(n)\) is reminiscent of the elliptic part of a trace formula over \(\SL_2\).
\end{itemize}
\end{remark}
\section{Smoothing and Poisson summation}
In this section, we study \(H_s(n)\) by first smoothing \(F_s\) and then applying Poisson summation. At present, the function \(F_s(x)\) is compactly supported and bounded, but it has singularities at \(x = \pm 1\). Although Poisson summation can technically be applied directly in this setting, doing so does not yield sufficiently strong bounds for the resulting Fourier transforms. Smoothing these singularities is therefore essential for obtaining the analytic estimates needed later. In particular, these bounds will be required to extend the region of convergence of \(\Res_{s=1} \: T_{\Ell,\geq0}^{k}(s_B,s)\). 
\\
Just like the proof for analytic continuation of the Riemann zeta function via theta series, we write the Zagier \(L\)-function as a Mellin transform of associated \(\theta\)-series (\ref{zagtheta}) to get
\begin{align*}
    &H_s(n) = \sum_{\substack{t \in \Z \\ t^2 - 4n \neq \square}} F_{s}\left( \frac{t}{2\sqrt{n}} \right) L\left(s, t^2 - 4n\right) \\
    &= \sum_{\substack{t \in \Z \\ t^2 - 4n \neq \square}} \frac{\pi^{(s+\delta_t)/2}}{\Gamma(\frac{s+\delta_t}{2})} F_{s}\left( \frac{t}{2\sqrt{n}} \right) \int_{0}^{\infty} \vartheta_{\Delta}(x) x^{(s+\delta_t)/2} d^{\times}x, \qquad \Delta = t^2 -4n
\end{align*}
Now using the change of variables \(x \mapsto \frac{x}{|t^2 -4n|}\), and then breaking the integral, we find that
\begin{equation}
    H_s(n) = \sum_{\substack{t \in \Z \\ t^2 - 4n \neq \square}}\frac{\pi^{(s+\delta_t)/2}}{\Gamma(\frac{s+\delta_t}{2})} F_{s}\left( \frac{t}{2\sqrt{n}} \right) |t^2 -4n|^{-\frac{s+\delta_t}{2}} \: \mathbf{I}(t^2-4n)   \label{eq:Hsn}
\end{equation}
where 
\[
\mathbf{I}(t^2 - 4n) = \int_{1}^{\infty} \vartheta_{\Delta} \left(\frac{1}{|t^2-4n|x}\right) x^{-\frac{s+\delta_t}{2}} d^{\times}x + \int_{1}^{\infty} \vartheta_{\Delta}\left(\frac{x}{|t^2-4n|}\right) x^{\frac{s+\delta_t}{2}} d^{\times}x
\]
We first simplify \(\mathbf{I}\) in the following proposition.
\begin{lemma} \label{compofI}
    The individual components of \(\mathbf{I}\) are given by 
    \begin{align*}
        &\int_{1}^{\infty} \vartheta_{\Delta} \left(\frac{1}{|t^2-4n|x}\right) x^{-\frac{s+\delta_t}{2}} d^{\times}x \\
        &= \sum_{k=1}^{\infty} v_{k,l} \: \chi_{D}\left(\frac{k}{v_{k,l}^2}\right) k^{s-1} \left(\frac{|t^2-4n|}{\pi}\right)^{-\frac{s}{2} + \frac{\delta_t + 1}{2}} \Gamma\left(-\frac{s}{2}+ \frac{\delta_t+1}{2}, \frac{\pi k^2}{|t^2-4n|}\right)
    \end{align*}
    and 
    \begin{align*}
        &\int_{1}^{\infty} \vartheta_{\Delta}\left(\frac{x}{|t^2-4n|}\right) x^{\frac{s+\delta_t}{2}} d^{\times}x \\
        &= \sum_{f^2 \mid \delta}^{'} f^{2\delta_t+1} \sum_{l=1}^{\infty} l^{\delta_t} \chi_{\frac{t^2-4n}{f^2}}(l) \left( \frac{|t^2-4n|}{\pi l^2f^4}\right)^{\frac{s+\delta_t}{2}} \Gamma\left(\frac{s+\delta_t}{2}, \frac{\pi l^2f^4}{|t^2-4n|}\right)
    \end{align*}
    where 
    \[
    \Gamma(z,x) = \int_{x}^{\infty} e^{-u} u^{z-1} du, \qquad x \in \R_{>0},\: z \in \C
    \]
\end{lemma}
\begin{proof}
    For the first integral, the result follows by applying the functional equation (\ref{eq:fneq}) and interchanging the summation and the integral. For the second integral, we directly apply the alternate form for the theta series (\ref{thetaseriesanotherform}) and interchange the summations and integrals. In both cases, the interchange is valid due to the decay of the exponential and the fact that the integral is away from 0 in both cases.
\end{proof}
Using the above lemma in \ref{eq:Hsn}, we get
\begin{align*}
    H_s(n) &= H_s^1(n) + H_s^2(n) \\
    H_s^1(n) &=\hspace{-0.1in} \sum_{\substack{t \in \Z \\ t^2 - 4n \neq \square}}\frac{\pi^{s-\frac{1}{2}}|t^2-4n|^{\frac{1}{2}-s}}{\Gamma(\frac{s+\delta_t}{2})} F_{s}\left( \frac{t}{2\sqrt{n}} \right) \sum_{k=1}^{\infty} v_{k,l} \: \chi_{D}\left(\frac{k}{v_{k,l}^2}\right) k^{s-1}  \Gamma\left(-\frac{s}{2}+ \frac{\delta_t+1}{2}, \frac{\pi k^2}{|t^2-4n|}\right)\\
    H_s^2(n)&= \hspace{-0.1in} \sum_{\substack{t \in \Z \\ t^2 - 4n \neq \square}} \frac{1}{\Gamma(\frac{s+\delta_t}{2})} F_{s}\left( \frac{t}{2\sqrt{n}} \right) \sum_{f^2 \mid \delta}^{'} \frac{1}{f^{2s-1}} \sum_{l=1}^{\infty} \frac{1}{l^s} \: \chi_{\frac{t^2-4n}{f^2}}(l) \:\Gamma\left(\frac{s+\delta_t}{2}, \frac{\pi l^2f^4}{|t^2-4n|} \right)
\end{align*}
In the next two propositions, we treat each of the terms \(H_s^{i}(n), \: i = 1,2\) and apply Poisson summation.
\begin{proposition} \label{H1sn}
    The term given by \(H_s^1(n)\) is equal to
    \[
    \frac{(4n)^{1-s} \pi^{s-\frac{1}{2}}}{4\Gamma(\frac{s+1}{2})} \sum_{f=1}^{\infty} \frac{1}{f^{3-2s}} \sum_{l=1}^{\infty} \frac{1}{l^{2-s}} \sum_{\xi \in \Z} \Kl_{l,f}(\xi,n) I_{l,f}(\xi,n)
    \]
    where
    \begin{equation}
        I_{l,f}(\xi,n) = \int_{-1}^{1} (1-y^2)^{\frac{1}{2}-s} F_s(y) \Phi_1\left(\frac{\sqrt{\pi} lf^2}{2 \sqrt{n(1-y^2)}} \right) e\left(\frac{-\xi \sqrt{n} y}{2lf^2}\right) dy \label{Ilfxin}
    \end{equation}
    with 
    \begin{equation}
        \Phi_1(x) = \Gamma \left(1-\frac{s}{2}, x^2\right) \label{eq:Phi1}
    \end{equation}
    and 
    \[
    \Kl_{l,f}(\xi,n) = \sum_{\substack{a \mod 4lf^2\\a^2-4n \equiv 0 \mod f^2 \\ 
    \frac{a^2 -4n}{f^2} \equiv 0,1 \mod 4}} \left( \frac{a^2-4n/f^2}{l}\right) e^{\frac{2 \pi i \xi a}{4lf^2}}, \qquad \left( \frac{\cdot}{\cdot}\right): \text{ Legendre symbol}
    \]
\end{proposition}
\begin{proof}
    Observe that by the simplification (\ref{Tellbroken}), the \(t\)-sum is finite and the \(k\)-sum converges due to the exponential, hence we can interchange the sums. Moreover, by the same simplification, whenever \(t^2-4n = \square\), \(F_s(t/2\sqrt{n}) =0\). Thus,
    \begin{align}
        H_s^1(n) = \pi^{s-\frac{1}{2}} \sum_{k=1}^{\infty} k^{s-1} \sum_{\substack{t \in \Z \\ t^2-4n = Dl^2}} v_{k,l} \chi_D\left(\frac{k}{v_{k,l}^2}\right) G(k,t) \label{H1snG}
    \end{align}
    with \[
    G(k,t) = \frac{|t^2-4n|^{\frac{1}{2}-s}}{\Gamma(\frac{s+\delta_t}{2})} F_{s}\left( \frac{t}{2\sqrt{n}} \right)  \Gamma\left(-\frac{s}{2}+ \frac{\delta_t+1}{2}, \frac{\pi k^2}{|t^2-4n|}\right)
    \]
    We first focus on the \(t\)-sum. Note that in this sum, \(k\) is fixed. Let \(k = s(k)^2 k_{\text{sf}}\), \(k_{\text{sf}}\) being square-free and let \(l = vb\). Then only \(b\) satisfying \((b,k/v^2) = 1\) appear in the sum. Hence, 
    \begin{align*}
        \sum_{\substack{t \in \Z \\ t^2-4n = Dl^2}} v_{k,l} \chi_D\left(\frac{k}{v_{k,l}^2}\right) G(k,t) = \sum_{b=1}^{\infty} \sum_{\substack{v^2 \mid k \\ (b, \frac{k}{v^2}) = 1}} \sum_{\substack{t \in \Z\\t^2 - 4n = D (kv)^2}} v \:\chi_{D}\left(\frac{k}{v^2}\right) G(k,t)
    \end{align*}
    Observe here that \(D\) denotes a fundamental discriminant. To get rid of this not-so-nice condition, we use Mobius inversion so that the above sum become
    \[
    \sum_{b=1}^{\infty} \sum_{\substack{v^2 \mid k \\ (b, \frac{k}{v^2}) = 1}} \sum_{\substack{t \in \Z\\t^2 - 4n \equiv 0 \mod b^2 v^2 \\ \frac{t^2-4n}{b^2v^2} \equiv 0 \mod 4 }} v \:\chi_d\left(\frac{k}{v^2}\right) G(k,t) \sum_{u^2 \mid \frac{t^2-4n}{b^2v^2}} \mu(u)
    \]
    Now interchanging the summation (which is valid because all the sums are finite by our simplification (\ref{Tellbroken}) ) and using the substitution \(bu=f\), we arrive at 
    \begin{equation}
        \sum_{v^2 \mid k} \sum_{f=1}^{\infty} \sum_{\substack{t \in \mathbb{Z}\\ t^2-4n \equiv 0 \mod f^2v^2\\ \frac{t^2-4n}{f^2v^2} \equiv 0,1 \mod 4}} v\: \chi_{\frac{t^2-4n}{f^2v^2}}\left(\frac{k}{v^2}\right) G(k,t) W(f) \label{tsum}
    \end{equation}
    where 
    \[
    W(f) =  \sum_{u \mid f} \mu(u) \chi_{u^2}\left(\frac{k}{v^2}\right)\: \mathbf{1}\{(\frac{k}{v^2}, \frac{f}{u}) = 1\ = \begin{cases}
        1 \qquad & f = 1\\
        0 & f> 1
    \end{cases}
    \]
    by a simple application of Mobius inversion again. Substituting (\ref{tsum}) into (\ref{H1snG}), we have that 
    \begin{align}
        H_s^1(n) &=  \pi^{s-\frac{1}{2}} \sum_{k=1}^{\infty} k^{s-1} \sum_{v^2 \mid k} \sum_{\substack{t \in \mathbb{Z}\\ t^2-4n \equiv 0 \mod v^2\\ \frac{t^2-4n}{v^2} \equiv 0,1 \mod 4}} v\: \chi_{\frac{t^2-4n}{v^2}}\left(\frac{k}{v^2}\right) G(k,t) \notag \\
        &= \pi^{s-\frac{1}{2}} \sum_{f=1}^{\infty} \sum_{l=1}^{\infty}  (lf^2)^{s-1} \: f \sum_{\substack{t \in \mathbb{Z}\\ t^2-4n \equiv 0 \mod f^2\\ \frac{t^2-4n}{f^2} \equiv 0,1 \mod 4}} 
        \chi_{\frac{t^2-4n}{f^2}}(l) \: G(lf^2,t) \label{finalhs1n}
    \end{align}
    where the equality follows by setting \( v= l\) and \(k = lf^2\). Now observe that \(\chi_{\frac{t^2-4n}{f^2}}(l)\) is periodic in \(t\) mod \(4lf^2\). Hence, the \(t\)-sum can be broken according to modulo class to get 
    \begin{align}
        \sum_{\substack{t \in \mathbb{Z}\\ t^2-4n \equiv 0 \mod f^2\\ \frac{t^2-4n}{f^2} \equiv 0,1 \mod 4}} 
        \chi_{\frac{t^2-4n}{f^2}}(l) \: G(lf^2,t) = \sum_{\substack{a \mod 4lf^2 \\ a^2 - 4n \equiv 0 \mod f^2 \\ \frac{a^2 -4n}{f^2} \equiv 0,1 \mod 4}} \chi_{\frac{a^2-4n}{f^2}}(l) \sum_{\substack{t \in \Z \\ t \equiv a \mod 4lf^2}} G(lf^2,t) \label{psfreadyhs1n}
    \end{align}
    The \(t\)-sum on the right is now ready for Poisson summation. Let us unfold \(G\) to comment on the smoothing:
    \[
    G(lf^2,t) = \frac{|t^2-4n|^{\frac{1}{2}-s}}{\Gamma(\frac{s+\delta_t}{2})} F_{s}\left( \frac{t}{2\sqrt{n}} \right)  \Gamma\left(-\frac{s}{2}+ \frac{\delta_t+1}{2}, \frac{\pi l^2f^4}{|t^2-4n|}\right)
    \]
    First, note that \(G(lf^2,t)\) is bounded and rapidly decreasing except possibly at \(t = \pm 2\sqrt{d}\). At \(t = \pm 2\sqrt{d}\), the incomplete Gamma factor is rapidly decreasing, and hence smooths out the singularities of \(F_s\). Therefore, the function \(G(lf^2,t)\) is a Schwartz function in the variable \(t\).\\ 
    Now, for a function \(F \in \mathcal{S}(\R)\), the Poisson summation formula gives 
    \begin{equation}
        \sum_{\substack{m \in \Z \\ m \equiv a \mod P}} F(m) = \frac{1}{P} \sum_{\xi \in \Z} \widehat{F}\left( \frac{\xi}{P}\right) \exp\left( \frac{\xi a}{P}\right) \label{PSF}
    \end{equation}
    Applying this formula to the \(t\)-sum in (\ref{psfreadyhs1n}), observing that \(F_s\) is only supported in \((-1,1)\) and substituting the result in (\ref{finalhs1n}) simplifies \(H_s^1(n)\)
    \[
    \frac{\pi^{s-\frac{1}{2}} (4n)^{1-s}}{4 \Gamma(\frac{s+1}{2})} \sum_{f = 1}^{\infty} \frac{1}{f^{3-2s}} \sum_{l=1}^{\infty} \frac{1}{l^{2-s}} \sum_{\xi \in \Z} \Kl_{l,f}(\xi,n) I_{l,f}(\xi,n) 
    \]
    as claimed.
\end{proof}
The term \(H_s^2(n)\) is treated in an analogous manner and is, in fact, technically simpler. We therefore state the resulting formula below without proof.
\begin{proposition} \label{H2sn}
The term \(H_s^2(n)\) is equal to 
\[
\frac{2 \sqrt{n}}{4 \Gamma(\frac{s+1}{2})} \sum_{f=1}^{\infty} \frac{1}{f^{2s+1}} \sum_{l=1}^{\infty} \frac{1}{l^{s+1}} \sum_{\xi \in \Z} \Kl_{l,f}(\xi,n) J_{l,f}(\xi,n)
\]
where 
\begin{equation}
    J_{l,f}(\xi,n) = \int_{-1}^{1} F_s(y) \Phi_2 \left( \frac{\sqrt{\pi} lf^2}{2\sqrt{n(1-y^2)}}\right) e\left(\frac{-\xi \sqrt{n}y}{2lf^2} \right) dy \label{Jlfxin}
\end{equation}
with 
\begin{equation}
    \Phi_2(x) = \Gamma\left( \frac{s+1}{2}, x^2 \right) \label{Phi2}
\end{equation}
and \(\Kl_{l,f}(\xi,n)\) defined as in \ref{H1sn}
\end{proposition}
We summarize the results of this section in the following theorem.
\begin{theorem}[Smoothing and Poisson summation] \label{smoothingPS}
For a fixed determinant \(n\), the term \(H_s(n)\) (\ref{Hsn}) is given by 
\[
\frac{2 \sqrt{n}}{4 \Gamma(\frac{s+1}{2})} \sum_{f=1}^{\infty} \frac{1}{f^{2s+1}} \sum_{l=1}^{\infty} \frac{1}{l^{s+1}} \sum_{\xi \in \Z} \Kl_{l,f}(\xi,n) \left( J_{l,f}(\xi,n) + \left(\frac{lf^2 \sqrt{\pi}}{2 \sqrt{n}}\right)^{2s-1} I_{l,f}(\xi,n)\right)
\]   
with the terms given as in 
(\ref{H1sn}) and (\ref{H2sn}).
\end{theorem}
With this, the elliptic part of our trace formula \(T^{k}_{\Ell, <0} (s_B,s)\) can be given by 
\begin{equation}
    \frac{\zeta(s)}{2 \Gamma(\frac{s+1}{2})}\sum_{d=1}^{\infty}\frac{1}{d^{s_B-\frac{ks}{2}}} \sum_{\substack{u \geq 1\\ u^2 \mid d^k}} \frac{C_k(d,u)}{u^s} \sum_{f,l=1}^{\infty} \frac{1}{(lf^2)^{s+\frac{1}{2}}} \sum_{\xi \in \Z}\frac{\Kl_{l,f}(\xi,\frac{d^k}{u^2})}{\sqrt{l}} H_{l,f}^s(\xi,\frac{d^k}{u^2}) \label{TellpostPS}
\end{equation}
with 
\begin{equation}
    H_{l,f}^s(\xi,n) = J_{l,f}(\xi,n) + \left(\frac{lf^2 \sqrt{\pi}}{2 \sqrt{n}}\right)^{2s-1} I_{l,f}(\xi,n)
\end{equation}
We now analyze the dominant term corresponding to \(\xi= 0\) in (\ref{TellpostPS}). 

\section{Dominant term of the elliptic part}
In this section, we give meromorphic continuation of the dominant term of the elliptic part given by 
\begin{equation}
    T^{k,\xi=0}_{\Ell, <0} (s_B,s) = \frac{\zeta(s)}{2 \Gamma(\frac{s+1}{2})}\sum_{d=1}^{\infty}\frac{1}{d^{s_B-\frac{ks}{2}}} \sum_{\substack{u \geq 1\\ u^2 \mid d^k}} \frac{C_k(d,u)}{u^s} \sum_{f,l=1}^{\infty} \frac{\Kl_{l,f}(0,\frac{d^k}{u^2})}{(lf^2)^{s+\frac{1}{2}}\sqrt{l}} \:H_{l,f}^s(0,\frac{d^k}{u^2}) \label{domterm}
\end{equation}
The first step will be writing \(H_{l,f}^s(0,n)\) as a Mellin transform. The following lemma does this.
\begin{lemma} 
    We can write \[
    H_{l,f}^s(0,n) = J_{l,f}(0,n) + \left(\frac{\sqrt{\pi} lf^2}{2 \sqrt{n}}\right)^{2s-1} I_{l,f}(0,n) 
    \] as 
    \begin{equation}
        H_{l,f}^s(0,n) = \frac{1}{2\pi i } \int_{(2)} \frac{(2w - \frac{3}{2})}{(w+\frac{s}{2}-1)(w-\frac{s+1}{2}) } \Gamma(w) \mathcal{F}(w,s) \left(\frac{\sqrt{\pi}lf^2}{2\sqrt{n}} \right)^{-2w+s+1} dw \label{Hlfsn}
    \end{equation}
    with 
    \begin{equation}
        \mathcal{F}(w,s) = \int_{-1}^{1} F_s(y) (1-y^2)^{w - \frac{s+1}{2}} dy \label{Fmathcal}
    \end{equation}
\end{lemma}
\begin{proof}
    Note that for \(\Phi_1(x)\) and \(\Phi_2(x)\) as defined in (\ref{eq:Phi1}) and (\ref{Phi2}) respectively, we have 
    \begin{align}
        \widetilde{\Phi}_1(z) = \frac{1}{z} \Gamma\left(\frac{z}{2} + 1 - \frac{s}{2} \right) ,\qquad \widetilde{\Phi}_2(z) = \frac{1}{z} \Gamma\left( \frac{z}{2} + \frac{s+1}{2}\right) 
    \end{align}
    In particular, for \(\mathfrak{R}(s) < 2\) (which we assume to be the case, we really are interested in the region near \(\mathfrak{R}(s) =1\)), \(\widetilde{\Phi}_1(z)\) and \(\widetilde{\Phi}_2(z)\) do not have poles in the region \(\mathfrak{R}(z) > 0\). Hence,
    \begin{align*}
        &\left( \frac{\sqrt{\pi} lf^2}{2\sqrt{n}}\right)^{2s-1} I_{l,f}(0,n) = \frac{1}{2\pi i}\int_{-1}^{1} (1-y^2)^{\frac{1}{2}-s} \int_{(2-\frac{s}{2})} \frac{\Gamma(\frac{w}{2}+1-\frac{s}{2})}{w} \left( \frac{\sqrt{\pi} lf^2}{2 \sqrt{n(1-y^2)}}\right)^{-w} dw \: dy \\
        &= \frac{1}{2\pi i} \int_{(2)} \frac{1}{(w+\frac{s}{2}-1)} \Gamma(w)\left(\frac{\sqrt{\pi}lf^2}{2\sqrt{n}} \right)^{-2w+s+1} \left( \int_{-1}^{1} F_s(y) (1-y^2)^{w - \frac{s+1}{2}} dy\right) dw
    \end{align*}
    where in the last line we have used the substitution \(w \mapsto 2w+s-2\). Note that the interchange is valid due to rapid decay of Gamma functions in vertical strips.\\
    Similarly, we have 
    \begin{align*}
        J_{l,f}(0,n) = \frac{1}{2 \pi i } \int_{(2)} \frac{1}{(w-\frac{s+1}{2})} \Gamma(w)\left(\frac{\sqrt{\pi}lf^2}{2\sqrt{n}} \right)^{-2w+s+1} \left( \int_{-1}^{1} F_s(y) (1-y^2)^{w - \frac{s+1}{2}} dy\right) dw
    \end{align*}
    and the claim follows.
\end{proof}
We now prove the main theorem below.
\begin{theorem}
    The dominant term of the elliptic contribution \[
    T^{k,\xi=0}_{\Ell}(s_B) = \Res_{s=1} T^{k,\xi=0}_{\text{ell}, < 0} (s_B,s)
    \]
    admits meromorphic continuation to \(\mathfrak{R}(s_B) \geq 0\) with a pole of order \(k\) at \(s_B = 1\).
\end{theorem}
\begin{remark}
We prove the theorem for the case \(k = 2h\). This assumption is made purely for computational convenience; the result remains valid for all \(k \in \Z_{>0}\).
\end{remark}
\begin{proof}
    By the above lemma, we have 
    \begin{align*}
        T^{k, \xi =0}_{\Ell, <0}(s_B,s) &=\frac{\zeta(s)}{2 \Gamma(\frac{s+1}{2})}\sum_{d=1}^{\infty}\frac{1}{d^{s_B-\frac{ks}{2}}} \sum_{\substack{u \geq 1\\ u^2 \mid d^k}} \frac{C_k(d,u)}{u^s} \sum_{f,l=1}^{\infty} \frac{1}{(lf^2)^{s+\frac{1}{2}}} \frac{\Kl_{l,f}(0,\frac{d^k}{u^2})}{\sqrt{l}} H_{l,f}^s(0,\frac{d^k}{u^2}) \\
        &= \frac{\zeta(s)}{2 \Gamma(\frac{s+1}{2})}\sum_{d=1}^{\infty}\frac{A(d,s)}{d^{s_B}}  
    \end{align*}
    where
    \[
    A(d,s) = \frac{d^{-\frac{k}{2}}}{2 \pi i } \sum_{u^2\mid d} u \; C_{k}(d,u) \sum_{l,f = 1}^{\infty} Kl_{l,f}(0,\frac{d^k}{u^2}) \: f \int_{(2)} \frac{(2z -\frac{3}{2})\Gamma(z) \mathcal{F}(z,s)}{(z + \frac{s}{2}-1)(z-\frac{s+1}{2})} \left( \frac{\sqrt{\pi} lf^2 u}{2\sqrt{d^k}}\right)^{-2z} dz
    \]
    For sufficiently large \(\mathfrak{R}(s_B)\), Perron's formula gives 
    \begin{align*}
        &\sum_{d \leq X} A(d,s) = \frac{1}{(2 \pi i)^2} \int_{(16h^2)} \frac{X^{w}}{w} \int_{(2)} \frac{(2z-\frac{3}{2})\Gamma(z)\mathcal{F}(z,s)}{(z+\frac{s}{2}-1)(z - \frac{s+1}{2})} \left( \frac{\sqrt{\pi}}{2}\right)^{-2z+s+1}
        \mathcal{D}(z,w)\:dzdw
    \end{align*}
    with 
    \begin{align*}
        \mathcal{D}(z,w) = \sum_{u=1}^{\infty} \frac{1}{u^{2z-1}N(u)^{w+\frac{k}{2}-kz}} \sum_{d=1}^{\infty} \frac{C_k(N(u)d,u)}{d^{w +\frac{k}{2}-kz}} \sum_{f,l=1}^{\infty} \frac{\Kl_{l,f}(0, \frac{N(u)d^k}{u^2})}{f^{4z-1} l^{2z}} 
    \end{align*}
    Using the description of \(\mathcal{D}(z,w)\) studied in the next section (\ref{Dw}), the above partial sum equals
    \begin{align*}
        \frac{2^{1-s} \pi^{\frac{s+1}{2}}}{(2 \pi i)^2} \int_{(16h^2)} \frac{X^{w}}{w} \int_{(2)} \frac{(2z-\frac{3}{2})\Gamma(z)\mathcal{F}(z,s)\zeta(4z-2)}{(z+\frac{s}{2}-1)(z - \frac{s+1}{2})\zeta(2z)} \left( \frac{\sqrt{\pi}}{2}\right)^{-2z} \prod_{m=-h}^h \zeta(w - m(2z-1)) dz dw
    \end{align*}
    The rest of the proof is obtained by shifting the \(z\)-contour to the right beyond the poles, interchanging the integrals and then shifting the \(s_B\)-contour to left of \(0\) while keeping track of the poles being crossed. Let's do this in steps. First, we shift the \(z\) contour far right past all poles, say \(\mathfrak{R}(z) = 8h\). In the region between \(\mathfrak{R}(z)=2\) and \(\mathfrak{R}(z) = 8h\),
    \begin{itemize}
        \item[-] There are no poles from \(\frac{(2z-\frac{3}{2})\Gamma(z)\mathcal{F}(z,s)\zeta(4z-2)}{(z+\frac{s}{2}-1)(z - \frac{s+1}{2})\zeta(2z)}\).
        \item[-] The poles of \(\prod_{m = -h}^h \zeta(s_B-m(2z-1))\) has poles at \[
        z = \frac{1}{2} + \frac{w-1}{2n}, \qquad n = 1,2,\dots h
        \]
        with \[
        \Res_{z = \frac{1}{2}+\frac{w-1}{2n}} \zeta(w- 2nz + n) = -\frac{1}{2n}
        \]
    \end{itemize}
    Thus,
    \begin{align}
    &\int_{(16h^2)} \frac{X^{w}}{w} \int_{(2)} \frac{(2z-\frac{3}{2})\Gamma(z)\mathcal{F}(z,s)\zeta(4z-2)}{(z+\frac{s}{2}-1)(z - \frac{s+1}{2})\zeta(2z)} \left( \frac{\sqrt{\pi}}{2}\right)^{-2z} \prod_{m=-h}^h \zeta(w - m(2z-1)) dz dw\notag \\
    &= \int_{(16h^2)} \frac{X^{w}}{w} \int_{(8h)} \frac{(2z-\frac{3}{2})\Gamma(z)\mathcal{F}(z,s)\zeta(4z-2)}{(z+\frac{s}{2}-1)(z - \frac{s+1}{2})\zeta(2z)} \left( \frac{\sqrt{\pi}}{2}\right)^{-2z} \prod_{m=-h}^h \zeta(w - m(2z-1)) dz dw \label{beforeinterchange}\\
    &+ 2 \pi i\sum_{n=1}^{h} R(n)
    \end{align}
    with \(R(n)\) given by
    \begin{equation}
        \int_{(16h^2)} \frac{X^w}{w} \frac{(2w-2-n) \Gamma(\frac{w-1+n}{2})\mathcal{F}(\frac{w-1+n}{2},s)\zeta(\frac{2w-2}{n})}{(w-1-n+ns)(w-1-ns)} \left(\frac{\sqrt{\pi}}{4}\right)^{\frac{1-w-n}{n}} \prod_{\substack{m = -h\\ m \neq n, n-1}}^{h} \zeta(\frac{wn-mw+m}{n}) du \label{Rn}
    \end{equation}
    We first treat (\ref{beforeinterchange}). Because of rapid decay of the Gamma factor and polynomial growth of zeta functions and \(\mathcal{F}\) in vertical strips, we can freely interchange the integrals, and shift the \(w\)-integral to \(\mathfrak{R}(w) = -\frac{1}{4}\). The only poles this picks up are at \(w =1\) (coming from \(\zeta(w-m(2z-1)\) at \(m=0\)) and \(w=0\) (from \(\frac{1}{w}\)). Hence, 
    the term (\ref{beforeinterchange}) is \[
    A(s)X + O(1)
    \]
    We shift the \(w\)-integrals in (\ref{Rn}) to \(\mathfrak{R}(w) = -\frac{1}{4}\). For this shift, in the region between \(\mathfrak{R}(w) = 7h\) to \(\mathfrak{R}(w) = -\frac{1}{4}\)
    \begin{itemize} \label{comment}
        \item[-] Pole of order at most 1 at \(w=0\)
        \item[-] Pole of order 1 at \(w = 1+ns\). Note that at \(s =1\), this pole is canceled by the zero of \(\mathcal{F}\) (that is \(\mathcal{F}(1,1) = 0\))
        \item[-] The pole at \(w = 1+\frac{n}{2}\) is canceled by the zero of \((2w-2-n)\).
        \item[-] A pole of order 1 at \(w = 1+ n(1-s)\). Note that at \(s =1\), this contributes an order to the pole at \(w=1\).
        \item[-] Pole at \(w=1\) of order \(2h-1\) coming from the product of zeta functions.
    \end{itemize}
    In summary, we have, for \(0< s< 2\) and \(s \neq 1\)
    \begin{align}
        \sum_{d \leq X} A(d,s) = \sum_{n=1}^{h} B(s)X^{1+ns} + \sum_{n=1}^{h} C(s) X^{1+n(1-s)} + XP_{2h-2}(X) + O(1)
    \end{align}
    Note that the pole from (\ref{beforeinterchange}) is engulfed in the pole of order \(2h-1\) in (\ref{Rn})
    At \(s = 1\), due to (\ref{comment}), we simply have
    \begin{align}
           \sum_{d \leq X} A(d,s) =  XP_{2h-1}(X) + O(1) \label{eq:semifinal}
    \end{align}
    Recall that 
    \begin{align*}
        T^{k,\xi = 0}_{\Ell} (s_B) &= \Res_{s=1} \frac{\zeta(s)}{2\Gamma(\frac{s+1}{2})} \sum_{d=1}^{\infty} \frac{A(d,s)}{d^{s_B}}\\
        &= \frac{1}{2}\sum_{d=1}^{\infty} \frac{A(d,1)}{d^{s_B}}
    \end{align*}
    (\ref{eq:semifinal}) now implies the claim.
\end{proof}
We conclude with the study of the Dirichlet series \(\mathcal{D}(z,w)\) in the next section.

\section{A Dirichlet series}
For complex numbers \(z, w \in \C\), define the Dirichlet series given by 
\begin{equation}
    \mathcal{D}(z,w) = \sum_{u=1}^{\infty} \frac{1}{u^{2z-1}N(u)^{w+\frac{k}{2}-kz}} \sum_{d=1}^{\infty} \frac{C_k(N(u)d,u)}{d^{w +\frac{k}{2}-kz}} \sum_{f,l=1}^{\infty} \frac{\Kl_{l,f}(0, \frac{N(u)d^k}{u^2})}{f^{4z-1} l^{2z}}  \label{DirichletD}
\end{equation}
where \[
N(u) = \prod_{q} q^{\lceil
 \frac{2 v_q(u)}{k} \rceil}
\]
As observed earlier, need the analytic properties of this series in treating the dominant term. \\
Observe that all the terms appearing in the Dirichlet series above are multiplicative and hence can be broken down into local factors. The following lemma does this for the double sum over \(f\) and \(l\).
\begin{lemma} \label{kldirichlet}
    For an integer \(n\),
    \[
    \textbf{K}(z, n) = \sum_{f,l = 1}^{\infty} \frac{\Kl_{l,f}(0,n)}{f^{4z-1} l^{2z}} = 4 \frac{\zeta(4z-2)}{\zeta(2z)} \prod_{q\mid n} \frac{(1-q^{-(2z-1)(v_q(n)+1)})}{(1-q^{-(2z-1)})} 
    \]
\end{lemma}
\begin{proof}
    Due to the multiplicative nature of \(\Kl_{l,f}\), the Dirichlet series breaks into Euler product 
    \begin{align}
        \textbf{K}(z,n) = \prod_{q} \textbf{K}_q(z,n) \label{prod}
    \end{align}
    with 
    \begin{align}
        \textbf{K}_q(z,n) = \sum_{u,v=0}^{\infty} \frac{\Kl_{q^v,q^u}(0,n)}{p^{u(4z-1)p^v(2z)}} \label{K_q}
    \end{align}
    Now, the Kloosterman sums \(\Kl_{q^v,q^u}(0,n)\) can be computed explicitly. One identity used multiple times ahead is Lemma 2 of Appendix A of \cite{Langlands2004}, given by 
    \begin{align}
        \sum_{a \mod q} \left( \frac{a^2-N}{q}\right) =1  \label{LemmaA}
    \end{align}
    Next, we compute this Kloosterman sums case-by-case for \(q \neq 2\). For \(q = 2\), the analysis is similar, but little more technically involved. Let \(r = v_q(n)\).
    \begin{itemize}
        \item[-] \(u=0,\: v=0\)\\
        Clearly, \(\Kl_{1,1}(0,n) = 1\)
        \item[-] \(u = 0, \: v > 0\)\\
        In this case
        \begin{align*}
            \Kl_{q^v,1} = \sum_{a \mod q^v} \left( \frac{a^2 - 4n}{q^v} \right) &= \begin{cases}
                \sum_{a_0 \hspace{-0.1in}\mod q} \left( \frac{a_0^2-4n}{q^v}\right) \sum_{a_1 \hspace{-0.1in} \mod q^{v-1}} 1 \qquad & r = 0\\
                \sum_{a \hspace{-0.1in} \mod q^v} \left( \frac{a^2}{q^v}\right) & r>0
            \end{cases}\\
            &= \begin{cases}
                q^{v} - q^{v-1}\left(1 +\left( \frac{n}{q}\right) \right) \qquad & r = 0, v \equiv 0 \mod 2\\
                -q^{v-1} & r = 0, v \equiv 1 \mod 2\\
                q^v - q^{v-1} & r > 0
            \end{cases}
        \end{align*}
        This follows from the observation that \(r>0 \implies q \mid n\), and when \(r=0\), in the case when \(v \equiv 1 \mod 2\), the result follows from \ref{LemmaA}.
        \item[-] \(r \geq 2u > 0, \: v= 0\)
        \begin{align*}
            \Kl_{1,q^u}(0,n) &= \sum_{\substack{a \mod q^{2u} \\ a^2 \equiv 4n \mod q^{2u}}} 1 = q^u
        \end{align*}
        \item[-] \(r \geq 2u > 0, \: v > 0\)
        \begin{align*}
            \Kl_{q^v,q^u}(0,n) &= \sum_{\substack{a \mod q^{v+2u}\\a^2 \equiv 4n \mod q^{2u}}} \left( \frac{a^2 - 4n/q^{2u}}{q^v}\right) \\
            &= \begin{cases}
                \sum_{a_0\hspace{-0.1in} \mod q^{v+u}} \left( \frac{a_0^2-4n_0}{q^v}  \right) \qquad & r \equiv 0 \mod 2,\: 2u = r,\: n = q^rn_0\\
                \sum_{a_0 \hspace{-0.1in} \mod q^{v+u}} \left( \frac{a_0^2}{q^v}\right) & r > 2u
            \end{cases} \\
            &= \begin{cases}
                q^{v+u} - q^{v+u-1}\left(1+\left(\frac{n_0}{q} \right)\right) \qquad & r \equiv 0 \hspace{-0.1in} \mod 2, \: 2u=r,\: v \equiv 0 \hspace{-0.1in} \mod 2\\
                -q^{v+u-1} & r \equiv 0 \hspace{-0.1in} \mod 2, \: 2u=r,\: v \equiv 1 \hspace{-0.1in} \mod 2 \\
                q^{v+u} - q^{v+u-1} & r > 2u
            \end{cases}
        \end{align*}
        This follows from the observation that \(r > 2u \implies \frac{a}{q^{2u}} \equiv 0 \mod q\). On the other hand, when \(2u = r\), we use \ref{LemmaA} for the case of \(v \equiv 1 \mod 2\).
        \item[-] \(2u > r\)\\
        Observe that in the case of \(r \equiv 1 \mod 2\), it is impossible to have \(a^2 \equiv 4n \mod q^{2u}\), and hence, \(\Kl_{q^v, q^u}(0,n) = 0\). Let's move to the case of \(r \equiv 0 \mod 2\) now. Again, when \(v = 0\), this becomes
        \begin{align*}
            \Kl_{1,q^u}(0,n) &= \sum_{\substack{a \mod q^{2u}\\ a^2 \equiv 4n \mod q^{2u}}} 1 = q^{r/2}\left( 1 + \left( \frac{n/q^r}{q}\right)\right)
        \end{align*}
        For the case of \(r \equiv 0 \mod 2\) and \(v > 0\), let \(n = q^rn_0\) and \(r = 2r_0\). In order to have \(a^2 \equiv 4n \mod q^{2u}\), we must have \(\left( \frac{n_0}{q}\right) = 1\). Let's assume that is the case moving forward and let \(u_j, \: j =1,2\) be the two roots of \(u_j^2 \equiv n_0 \mod q^{2u-r+1}\). Then the Kloosterman sum can be written as
        \begin{align*}
            \Kl_{1,q^u}(0,n) &= \sum_{\substack{a \mod q^{v+2u}\\ a^2 \equiv 4n \mod q^{2u}}} \left( \frac{a^2 - 4n/q^{2u}}{q^v}\right)\\
            &= q^{v-1} \sum_{\substack{a_0 \mod q^{1+2u}\\a_0^2 \equiv 4n \mod q^{2u}}} \left( \frac{a_0^2-4n/q^{2u}}{q^v}\right)\\
            &= q^{v-1} \sum_{\substack{a_1 \mod q^{1+2u-r_0}\\ a_1^2 \equiv 4n/q^r \mod q^{2u-r}}} \left( \frac{a_1^2-4n_0/q^{2u-r}}{q^v}\right) \\
            &= q^{v+r_0-1} \sum_{j=1,2} \sum_{a_3 \mod q} \left(\frac{a_3u_j}{q^v} \right) \\
            &= q^{v+r_0}\left(1-\frac{1}{q} \right)\left( 1 + \left( \frac{n_0}{q}\right)\right)
            \begin{cases}
                    0 \qquad & v \equiv 1 \mod 2\\
                    1 & v \equiv 0 \mod 2
            \end{cases}
        \end{align*}
    \end{itemize}
    Now, the result follows by calculating the sum \ref{K_q} and then substituting into the Euler product \ref{prod}. 
\end{proof}
We now use this lemma and study the series \(\mathcal{D}(w)\) in the following proposition.
\begin{proposition} \label{Dw}
    For \(w, z \in \C\) and \(k = 2h\), the Dirichlet series \(\mathcal{D}(z,w)\) (\ref{DirichletD}) is given by 
    \[
    \mathcal{D}(z,w) = 4 \frac{\zeta(4z-1)}{\zeta(2z)} \prod_{m=-h}^{h} \zeta(w - m(2z-1))
    \]
\end{proposition}
\begin{proof}
    By Lemma (\ref{kldirichlet}), it is enough to show that 
    \[
    \mathcal{G}(z) = \prod_{m=-h}^{m = h} \zeta(w - mz)
    \]
    where 
    \[
    \mathcal{G}(z) = \sum_{u=1}^{\infty} \frac{1}{u^{z} N(u)^{w -\frac{ks}{2}}} \sum_{d=1}^{\infty} \frac{C_k(N(u),d)}{d^{w-\frac{kz}{2}}} \prod_{q \mid \frac{N(u)^k}{u^2} d^k} \frac{1-q^{-z(v_q\left( \frac{N(u)^k d^k}{u^2}\right)+1)}}{1-q^{-z}}
    \]
    Now, since 
    \[
    C_k(d,u) = \prod_{q} \mathfrak{c}(k, v_q(d), v_q(u))
    \]
    is multiplicative, we can write \(\mathcal{G}(z)\) as an Euler product
    \begin{equation}
        \mathcal{G}(z) = \prod_q \mathcal{G}_q(z) \label{Geuler}
    \end{equation}
    where for \(r = v_q(u), m = v_q(d)\) and \(n = \lceil \frac{2r}{k} \rceil = \lceil \frac{r}{h} \rceil\), the local factors are given by
    \[
    \mathcal{G}_q(z) = \sum_{m,r \geq 0} q^{-rz -(m+n)(w - hz)} \mathfrak{c}(2h ,m+n, r) \frac{1- q^{-z(2hn+2hm-2r+1)}}{1-q^{-z}}
    \]
    Let us set \(M = m+n\) and \(X = q^{-z}\). Observe that because \(\mathfrak{c}(k,m+n,r)\) is supported only for \(0 \leq r \leq hM\), the \(r\)-sum is in fact finite, and we can rewrite the local factor as 
    \begin{align*}
        \mathcal{G}_q(z) = \sum_{M=0}^{\infty} q^{-Mw} \sum_{r=0}^{hM} \mathfrak{c}(2h, M ,r) S_X(hM-r)
    \end{align*}
    with 
    \[
    S_X(a) = \sum_{i=-a}^{a} X^i
    \]
    Recall that the coefficients \(\mathfrak{c}(2h,M,r)\) (\ref{eq:coeff}) are given by
    \begin{align*}
        \mathfrak{c}(2h,M,r) &= p_{M,2h}(r) - p_{M,2h}(r-1), \qquad p_{M,2h}(-1) =0 \\ &\sum_{r=0}^{2hM} p_{M,2h}(r)q^{r} = \prod_{j=1}^{2h} \frac{1-q^{M+j}}{1-q^{j}} = \qbinom{M+2h}{2h}_q 
    \end{align*}
    Observe that by definition, we also have the following symmetry
    \begin{align*}
        p_{M,2h}(r) = p_{M,2h}(2hM-r)
    \end{align*}
    Using the telescoping nature of the coefficients \(\mathfrak{c}(2h,M,r)\) and the above symmetry, we have 
    \begin{align*}
         \mathcal{G}_q(z) &= \sum_{M=0}^{\infty} q^{-Mw} X^{-hM} \prod_{j=1}^{2h} \frac{1-X^{M+j}}{1-X^j} \\
         &= \sum_{M=0}^{\infty} \qbinom{M+2h}{2h}_X (X^{-h}q^{-w})^M \\
         &= \prod_{m=-h}^{h} (1-q^{-(w-mz)})^{-1}
    \end{align*}
    where the last equality comes from the \(q\)-binomial identity 
    \[
    \sum_{M=0}^{\infty} \qbinom{M+N}{N}_X t^M = \prod_{j=0}^{N} \frac{1}{1-X^jt}
    \]
    The lemma now follows from (\ref{Geuler}).
\end{proof}
 
\pagebreak

\bibliographystyle{amsalpha}
\bibliography{references}

\end{document}